\documentclass[12pt]{article}

\usepackage[english]{babel}
\usepackage[utf8x]{inputenc}		
\usepackage{enumerate}
\usepackage{textcomp}               
\usepackage{typearea}
\usepackage{pdfpages}
\usepackage[toc]{appendix}
\usepackage[]{todonotes}
\usepackage{subcaption}
\usepackage{verbatim}
\usepackage{comment}
\usepackage{mathtools}
\mathtoolsset{showonlyrefs,showmanualtags}
\usepackage[hidelinks]{hyperref}

\usepackage{bm}
\usepackage{enumitem}
\usepackage{amsfonts, amsmath, amssymb, amsthm, dsfont}
\usepackage{amsthm}			
\usepackage{thumbpdf}	
\usepackage{natbib}
\allowdisplaybreaks

\usepackage{booktabs}

\usepackage[noblocks]{authblk}
\usepackage{ulem}

\usepackage{pgf, tikz}
\usepackage{tikz-cd}		
\usepackage{bbm}

\theoremstyle{definition}
\newtheorem{definition}{Definition}[section]

\newtheorem*{assumption*}{Assumption}

\newtheorem*{condition*}{Condition}

\theoremstyle{plain}
\newtheorem{theorem}[definition]{Theorem}

\newtheorem{lemma}[definition]{Lemma}

\theoremstyle{remark}

\makeatletter
\patchcmd{\proof@init}{\@empty}{Proof\@ifnotempty{#1}{\ #1}}{}{}
\makeatother

\usepackage{graphicx}

\newcommand{\N}{ \mathbb{N} }

\newcommand{\R}{ \mathbb{R} }

\newcommand{\calB}{\mathcal{B}}

\newcommand{\calD}{\mathcal{D}}
\newcommand{\calE}{\mathcal{E}}
\newcommand{\calF}{\mathcal{F}}
\newcommand{\calG}{\mathcal{G}}

\newcommand{\calS}{\mathcal{S}}

\newcommand{\calZ}{\mathcal{Z}}
\newcommand{\matnull}{{0}}

\newcommand{\eins}{{\bm 1}}
\newcommand{\matA}{{\bm A}}
\newcommand{\matB}{{\bm B}}
\newcommand{\matC}{{\bm C}}
\newcommand{\matD}{{\bm D}}

\newcommand{\matX}{{\bm X}}

\newcommand{\matS}{{\bm S}}

\newcommand{\vecnull}{{\bm 0}}

\newcommand{\veca}{{\bm a}}
\newcommand{\vecb}{{\bm b}}
\newcommand{\vecB}{{\bm B}}
\newcommand{\vecc}{{\bm c}}

\newcommand{\vecd}{{\bm d}}

\newcommand{\vecg}{{\bm g}}

\newcommand{\vecm}{{\bm m}}

\newcommand{\vecp}{{\bm p}}

\newcommand{\vecu}{{\bm u}}

\newcommand{\vecS}{{\bm S}}

\newcommand{\vecv}{{\bm v}}

\newcommand{\vecx}{{\bm x}}
\newcommand{\vecX}{{\bm X}}
\newcommand{\vecy}{{\bm y}}

\newcommand{\vecZ}{{\bm Z}}
\newcommand{\vecz}{{\bm z}}
\newcommand{\vecU}{{\bm U}}

\newcommand{\bfmu}{\bm \mu}

\newcommand{\bftau}{\bm\tau}
\newcommand{\bfDelta}{\bm\Delta}

\newcommand{\bfSigma}{\bm\Sigma}

\newcommand{\vecop}{ \operatorname{vec }}

\newcommand{\EE}{\mathbb E}

\newcommand{\Cov}{{\mbox{Cov\,}}}

\newcommand{\matid}{I}

\newcommand{\id}{\operatorname{id}}

\excludecomment{proofcomplete}
\excludecomment{comment}

\title{Adaptive Thresholds for Monitoring and Screening in Imbalanced Samples: Optimality and Boosting Sensitivity} 

\author[$\dagger$]{Ansgar Steland}

\affil[$\dagger$]{\footnotesize RWTH Aachen University\\ Insitute of Statistics and AI Center}

\date{}

\begin{document}
	
	\maketitle

\begin{abstract}
Suppose (standardized) measurements or statistics are monitored to raise an alarm when a threshold is exceeded. Often, the underlying population is heterogenous with respect to important discrete variables and thus samples may consist of imbalanced classes. We propose to use thresholds which depend on such covariates to boost the sensitivity for rare classes, which otherwise tend to be ignored. Under mild conditions, we identify optimal threshold functions and develop a feasible procedure for their computation. Further, for the proportional rule a nonparametric estimator of the threshold function is proposed and a central limit theorem is shown, including the case that conditional mean and variance used for standardization are estimated. For feasible uncertainty quantification a bootstrap scheme is proposed. The approach is illustrated and evaluated by a real data analysis.  

\end{abstract}

\section{Introduction}

A decision framework is considered where univariate observations (or summary statistics) of a sequential data stream are thresholded to accept or reject a null hypothesis against a change alternative hypothesis, the first $n$ points being observed and reserved as a learning sample.  \color{black} This setting hosts several classical statistical problems including screening of populations for diseases, \cite{BlackWelch1997}, monitoring a production process by a control chart, \cite{Montgomery2021ISQC}, and sequential detection of changes in parameter estimates, \cite{CsorgoHorvath1997} and \cite{Steland2007WDF}. \color{black} Often, however, there is additional (external) information given by a modifier variable, $Z$, about the framework or environment, which should  be taken into account when standardizing measurements and determining a suitable decision threshold. Among the diverse areas, where such information is available, are intense-care monitoring of patients, where the monitoring rule should be individualized according to the patient's state-of-health and the decision rule should be more sensitive for rare but critical states, and screening of (sub-) populations to identify cases which are likely to develop a disease, where variables such as the body mass index may have predictive power and special attention should be paid to cases with high risk factor values. The adaptation of monitoring thresholds by ad-hoc rules is common practice in vital sign monitoring of postsurgical patients and in general care units, \cite{WelchEtAl2016}. In a recent study, \cite{van_Rossum_2021} investigated and compared various approaches  to take account of personal and situational factors, mainly to avoid too many false alarms and balance them with the sensitivity to detect adverse events. For a discussion of further potential areas of applications see \cite{SteRafa2024}. 

We study threshold-based decision rules, as arising in hypothesis testing, designed as statistical tests and thus controlling the type I error rate of falsely rejecting the null hypothesis and declaring an alarm. However, contrary to the classical setting, the additional information (modifier) $Z$ is taken into account by modeling the decision threshold as a function of the modifier $ Z $. In this way $Z$ directly affects the decision and thus the probability to raise an alarm (i.e. reject the null hypothesis). We focus on the case of categorical $Z$ with a special emphasis on the imbalanced classes problem. Specifically,  we study a rule suggested by \cite{SteRafa2024}, the {\em proportional rule}, and a generalization thereof, the $ \gamma $-proportional rule, which distribute  the significance level over the sample space $ \calZ $ of $Z$, in such way that the statistical power is  larger for small (minority) classes than for large classes. This is motivated by the fact that often small classes  represent risky observations, for which it is more likely that future data  belong to the alternative hypothesis due to a change of the underlying distributions. 

Under mild conditions, the proportional rule and the related $ \gamma $-proportional rule  belong to the class of admissible threshold functions induced by a discrete subprobability measure on $ \calZ $. We establish sufficient conditions for threshold functions ensuring that the corresponding decision rule is optimal. Here optimality means that one aims at determining a level $\alpha $ rule with prespecified conditional type II error rates, which maximizes the detection power of a selected minority class. Of course, this requires to fix a distribution for the alternative for which the rule is suitably constructed. In applications, however, it might be difficult to specify all these parameters. 

Since the proportional rule depends on the distribution function, $ \Psi $, of standardized measurements, which was assumed to be known in \cite{SteRafa2024}, this paper tackles the  more involved problem that $ \Psi $ is unkown and thus needs to be estimated.  A nonparametric estimator based on nonparametric quantile estimates is proposed and its asymptotic distribution theory is established by viewing it as a differentiable statistical functional. The results are extended to the relevant case that the class-wise conditional means and standard deviations need to be estimated to standardize measurements. Here, suitable results for the residual empirical process are provided. Since the asymptotic distribution is intractable, we propose to use the bootstrap to assess the uncertainty of the estimated thresholds. For that purpose, a bootstrap central limit theorem is provided which leads to a feasible resampling scheme. 

The results are illustrated by analyzing a medical dataset about diabetes mellitus. The body mass index (BMI) is used as potential risk factor to define a class of high-risk patients for which the decision rule should be more sensitive than for the other patients. It turns out that the proposed rule improves upon the standard choice of a constant threshold for the standardized measurements. By a data driven simulation study, which mimics screening/monitoring of a large population by a rule estimated from a much smaller learning sample, we further assess the accuracy of the estimated thresholds in terms of the type I error rates.

The rest of the paper is organized as follows. Section~\ref{Sec: Method} provides the details of the framework and introduces the proposed decision rules. Optimality results are given in Section~\ref{Sec: Optimality}. The nonparametric estimator and its asymptotic distribution theory of the proportional rule is provided in Section~\ref{Sec:PropThreshold}. Extensions to the case that the standardization of $ U_t $ is based on estimators are given in Section~\ref{Sec:Residuals}. The bootstrap procedure is described in Section~\ref{Sec:Bootstrap}. Section~\ref{Sec: Example} provides a real data analysis including simulations to illustrate the method and assess its properties. Proofs and additional results are provided in an appendix, see Section~\ref{Sec: Proofs}.

\section{Methodology}
\label{Sec: Method}

We observe a potentially infinite sequence, $ (U_t,Z_t) $, $ t \ge 1 $, of pairs of
statistics $U_t $ and additional environment information $ Z_t $, \color{black} both attaining values in the real numbers and \color{black} defined on a common probability space. It is assumed that the first $n$ observations $(U_t, Z_t) $, $ 1 \le t \le n $,  represent a learning sample satisfying a no-change null hypothesis $ H_0 $ under which it forms a random sample. This is an extension of the monitoring framework addressed to \cite{ChuEtAl996}, but we take a different view on the problem and are specifically interested in designing procedures which take account of $ Z_t $ going beyond common standardization by a regression model approach and detecting structural instability. 
Usually, the $U_t$ are obtained by standardizing raw measurements $ X_t $ under $H_0$, but they may also be standardized summary statistics of samples drawn at each $t$. Despite this possible and useful extension, we name the $U_t $ observations in the sequel. Focusing on the one-sided case, these observations and future ones are analyzed and marked as suspicious, if $ U_t $ is too large in view of  $ H_0 $ and instead speaks in favor of an alternative hypothesis $ H_1 $  under which (without loss of generality) the upper tail probability of $ U_t $ is larger than under $ H_0 $, although this often may apply only for a subpopulation determined by certain values of $ Z_t $. Thresholded observations are regarded suspicious and potentially belonging to the subpopulation for which the distribution changes. Throughout, we assume that the first $n$ data points form a no-change learning sample and are available when setting up the procedure. 

In a monitoring setup further observations, $ X_t $, $ t  >  n$, are ordered in time and arrive one after the other. The first time point $t^* > n $ where an alarm is raised can then be used to estimate the change-point where the distribution  changes.  In monitoring, $ X_t $ may be a suitable change-point statistic of underlying raw measurements, $ \xi_t $, e.g., a scaled cumulated sum (CUSUM) $ X_t =  h^{-1/2} \sum_{j=0}^{h-1} \xi_{t-j} $, 
for some fixed window length $ h \in \N $. 
It is well known that such CUSUM statistics react quickly to level changes, and the proposed methodology allows for this setting as well. In a screening setup, which we also have in mind, the data for $ t > n $ represents a cross-sectional sample that is collected and analyzed. Here, one is interested in testing  each observation by marking it as suspicious or not, assuming that this may indicate that the conditional law given $ Z_t $ of the screening sample differs from the learning sample.

We will determine the threshold function $ c(\cdot) $ in such a way that a prespecified type I error rate, $ \alpha \in (0,1) $, of a false alarm is maintained for each observation. In the monitoring setup, the no-change average run length until a signal is then $ 1/\alpha $ under independence, of course. However, we focus on type I and type II error rates instead of average run lengths, which are more relevant characteristics for applications where, instead of quick detection, statistical guarantees matter that an alarm is significant in the classical sense of hypothesis testing. Moreover, we also focus on the construction of the rule and its estimation from the learning sample, leaving results about its sequential behaviour to future research.

In this work, the case of discrete-valued nominal $ Z_t$ taking values in a finite set $ \calZ = \{ z_1, \ldots, z_K \} $ for some $K \in \N $ is considered, such that the population is partitioned in $K$ classes.  \color{black} To keep the presentation simple, we introduce the approach for real-valued $ Z_t $ such that $ \calZ \subset \R $, but the theoretical results of Section~\ref{Sec:PropThreshold} and Section~\ref{Sec:Residuals} allow for the multivariate case where $\calZ \subset \R^q $ for some $ q \in \N $.  If $ Z_t $ is not a categorical variable, one can proceed by discretizing it. \color{black} Denote the class probabilities by $ p_k $, $ 1 \le k \le K $, and put $ \vecp  = (p_1,\ldots, p_K)^\top $.  These classes may or may not be closely related to the classes considered in a classification framework, where one assumes that observations from different classes have different distributions. Further, in classification the classes are defined in terms of a true variable (e.g. disease status), whereas in our setup $ Z_t $ is regarded as a predictive variable for the true status which is not available for construction of the decision rule. 
Typically, the classes defined by $ Z_t $ are imbalanced, which can lead to severe bias problems as small classes may have a negligible effect on statistical quantities such as the type I and type II error rates or misclassification rates. Clearly, the learning sample is partitioned as well, namely in subsamples corresponding to the values $ z_k $ observed in the learning sample, and unbalanced class probabilities give rise to imbalanced subsamples.  In this paper, we are specifically interested in rules tailored to such imbalanced settings.

We assume a regression model \color{black} for the raw observation \color{black} 
\[
X_t =\left\{ 
\begin{array}{cl}
	\mu(Z_t) + \sigma(Z_t) U_t,  &\qquad 1 \le t < t^\dagger, \\
	\mu^\dagger(Z_t) + \sigma^\dagger(Z_t) U_t, & \qquad t^\dagger \le t < T,
\end{array}
\right.
\]
for real-valued functions $ \mu(\cdot), \mu^\dagger(\cdot) $ and positive functions $ \sigma( \cdot), \sigma^\dagger(\cdot) $ with 
\[ 
	|\mu(\cdot) - \mu^\dagger(\cdot) |+|\sigma(\cdot) - \sigma^\dagger(\cdot) | \not= 0, \qquad  \text{on\ }  [t^\dagger,T], 
\] 
and i.i.d. random variables $ U_t \equiv \frac{X_t - \mu(Z_t)}{\sigma(Z_t)} $ distributed according to a c.d.f. $ \Psi $ not depending on $ \mu( \cdot) $ and $ \sigma( \cdot ) $, where $ U_t $ and $ Z_t $ are assumed to be independent. $T \in \N \cup \{ \infty \} $ is the time horizon. $ t^\dagger $ is the change-point where the distribution of $ X_t $ changes. Occasionally, we consider generic variables $ X, Z $ or $U$, whose distribution can be inferred from the context. 

The decision procedure is now as follows: Observation $ U_t $ is marked suspicious, if  
\[
	U_t>c(Z_t)
\]
for some threshold function $ c( \cdot ) $. \color{black} In a monitoring scenario, we aim at detecting all time points suspicious for a change and label all observations  $ U_t $ exceeding $ c(Z_t) $. An alarm is raised at the first time instant where this occurs, and usually this time point is regarded as an estimator of $ t^\dagger $. In a screening scenario, the screening may take place before or after the change point $t^\dagger $, and the goal is to mark all suspicious observations and analyze the corresponding classification result. \color{black}

We consider threshold functions which ensure that 
the false-alarm probability does not exceed a given significance level $ \alpha \in (0,1) $, i.e., the type I error rate constraint,
$$
p_f = P( U_1>c(Z_1)) \le \alpha,
$$
is satisfied. Thus, for each observation the decision is a statistical level $ \alpha $ test. 
In this general formulation, there are infinitely many admissible functions $ c(\cdot) $. The classical approach is as follows: For known $ \mu( \cdot ) $ and $ \sigma( \cdot ) $ put $ U_t = \frac{ X_t - {\mu}(Z_t)}{\sigma(Z_t) } $ and use the constant threshold $ c(z) = \Psi^{-1}(1-\alpha) $. In terms of the measurements $ X_t $, this leads to the threshold function $  t(z) = \mu(z) +  \sigma(z) \Psi^{-1}(1-\alpha)  $. If $ \mu(\cdot) $ and $ \sigma(\cdot) $ are unkown, one uses suitable estimators $ \hat \mu_n( \cdot) $ and $ \hat\sigma_n(\cdot) $ calculated from the learning sample. In this way, the information $ Z_t $ is used, but in terms of the standardized values $ U_t $ resp. $ \hat U_t $ the same threshold is used for all categories $ z \in \calZ$. Since 
\[p_f = \sum_{k=1}^K p_k P(U_1 > c(z_k)|Z_1=z_k),\] 
the imbalanced classes problem may arise that categories with small $p_k $ (minority classes) are more or less ignored by a rule $ c(\cdot) $, since their contribution to the type I error rate is small or even negligible, and the same applies to the probability to raise an alarm under alternative hypotheses. 

\subsection{Proposed decision rules}

To mitigate the issues arising for imbalanced probabilities, one may approach the problem by assigning smaller thresholds to classes with small class probabilities.  One can determine optimal thresholds as discussed in the next section, which requires to specify a suitable target alternative distribution and to set up conditional type II error rates for all categories. But in applications, this information is rarely available. Therefore, we study a threshold function which automatically assigns smaller thresholds to rare classes and larger ones to classes which are more frequently observed.

Note that the smaller $ c(z_k) $ for some possible value $ z_k $ of $ Z_t $, the higher the sensitivity of the rule for the alternative hypotheses $ H_1 $ that $ E(X_t|Z_t=z_k) $ exceeds the $ H_0 $-value $ \mu(z_k) $ within the class defined by $ Z_t = z_k $. This suggests to define $ c(z_k) $ as a monotone function of $ p_k = P(Z_t=z_k)$. An interesting rule of this type, introduced in \cite{SteRafa2024} and named proportional rule, is  defined as
\begin{equation}
	\label{ProportionalRule}
	c_{\text{prop}}(z_k) = \Psi^{-1}\left( \frac{(1-\alpha)p_k}{\sum_{j=1}^K p_j^2 } \right), \qquad 1 \le k \le K,
\end{equation}
or, more generally, for some fixed $ 0 < \gamma \le 1 $ the $ \gamma $-proportional rule
\begin{equation*}
	\label{GammaProportionalRule}
	c_{\text{prop},\gamma}(z_k) = \Psi^{-1}\left( \frac{(1-\alpha)p_k^\gamma}{\sum_{j=1}^K p_j^{\gamma+1} } \right), \qquad 1 \le k \le K,
\end{equation*}
provided the probability distribution $ \vecp = (p_1, \ldots, p_k)^\top \in[0,1]^K $ of $ Z_t $ is admissible in the sense that it ensures that the rule is defined, i.e., $ p_k^\gamma < (1-\alpha)^{-1} \sum_{j=1}^K p_j^{\gamma+1} $, for all $k$. Clearly, the smaller $ \gamma $, the smaller the range of the $ p_k^\gamma $, which leads to a larger set of admissible distributions $ \vecp $.
Note that for $ p_k = 1/K $, $ 1 \le k \le K $, the constant threshold $ \Psi^{-1}(1-\alpha) $ is assigned to each class.  Otherwise, classes with small $p_k$ receive a smaller threshold. By construction,
\[
p_f = \sum_{k=1}^K p_k [1-\Psi( c_{\text{prop},\gamma}(z_k ) )] = \alpha,
\]
i.e., the resulting detector maintains the nominal type I error rate of a false alarm.

One can easily check that for the binary case ($ K = 2 $) the proportional rule always exists, if $ \alpha \le \min(p_1,p_2) $, see \cite{SteRafa2024}, and for $K > 2 $ one may enlarge the scope of distributions by dichotomizing the weights, see appendix for details.

\section{Optimal threshold functions}
\label{Sec: Optimality}

The  rules $ c_{\text{prop}} $ and $ c_{\text{mod}} $ are special cases of the more general class of solutions due to \cite{SteRafa2024}, which attain the form
\begin{equation}
	\label{GeneralClass}
	c(z_k) = \Psi^{-1}( g_k/p_k) \eins_{\{p_k>0\}} + \Psi^{-1}(0) \eins_{\{p_k=0\}}, \qquad 1 \le k \le K,
\end{equation}
for a $ \vecp $-subprobability vector $ \vecg = (g_1, \ldots, g_K)^\top $, i.e., $ 0 \le g_k < p_k $, $ 1 \le k \le K $, satisfying the type-I error rate constraint $ \sum_{k=1}^K g_k \ge 1-\alpha $. Specifically, the proportional rule corresponds to the choice $ g_k = (1-\alpha) p_k^2 / \sum_{j=1}^K p_j^2 $, $ 1 \le k \le K $. 

Thus, we consider the problem to select the threshold function $c(\cdot) $ in such a way that the (conditional) detection power is at least $ 1- \beta $ (resp. $1-\beta_k$), for some prespecified marginal type II error rate $ \beta \in (0,1) $ and conditional type II error rates $ \beta_k $, respectively, when the observations follow a different model. For that purpose, we assume that under $ H_1 $, given $ Z_t = z $, the observations $ U_t $, $ t > t^* $, are distributed according to
\begin{equation}
	\label{H1Model}
	U_t  \sim \Delta(z) + \Sigma(z) U_t^0, \qquad U_t^0 \sim \Psi,
\end{equation}
for some measurable real-valued function $ \Delta( \cdot) $ and a measurable positive function $ \Sigma(\cdot) $ with $ \Delta \not\equiv 0 $ or $ \Sigma \not\equiv 1 $. $ t^* $ is called change-point. If $ t^* = 1 $, then we are given the classical hypothesis testing problem. Since the results of this section are point-wise, we confine ourselves to this case. Note that the null hypothesis $ H_0$ is given by $ H_0 : \Delta \equiv 0, \Sigma \equiv 1 $.

For simplicity of presentation and without loss of generality, the effect of $ Z_t $ is formulated in terms of the law of $ U_t $ instead of $ X_t $, and in agreement with the specification of the decision rule, we shall focus on upper one-sided deviations and therefore assume that $ \Delta(\cdot) > 0 $. To this end, note that if we specify the alternative by assuming that $ X_t $ given $ Z_t = z $ has conditional mean $ \tilde{\Delta}(z) $,  conditional variance $ \tilde{\Sigma}(z) $ and $ \frac{X_t - \tilde{\Delta}(z)}{\tilde\Sigma(z) } \sim \Psi $ (given $ Z_ t = z $) holds, then one readily checks that \eqref{H1Model} holds with $ \Delta(z) = \frac{\tilde{\Delta}(z)-\mu(z)}{\sigma(z)} $ and $ \Sigma(z) = \frac{\tilde{\Sigma}(z)}{\sigma(z)} $.

\subsection{Sufficient conditions for optimality}

To ensure the type I error rate constraint, $ \sum_k g_k \ge 1-\alpha $, the sub-probabilities $ g_k $ and thus the threshold values, $ c(z_k) $, need to be large enough. Contrary, to ensure a certain minimal detection power, the threshold function and thus the $ g_k$ need to be selected small enough. The following theorem collects sufficient conditions on a threshold function. 

\begin{theorem} 
	\label{ThOptimality}
	Suppose that model \eqref{H1Model} holds true. 
	\begin{itemize}
		\item[(i)] Any threshold function $ c(\cdot ) $ satisfying
		\begin{align}
			\label{H0Constraint}
			\sum_{k=1}^K \Psi( c(z_k) ) p_k &\le \alpha, \\ 
			\label{H1Constraint}
			c(z_k) \le c_{\text{opt}}^* (z_k) &= \Delta(z_k) + \Sigma(z_k) \Psi^{-1}( \beta ) , \qquad 1 \le k \le K,
		\end{align}
		is a level $ \alpha $ rule with detection power at least $ 1- \beta $. If 
		\begin{equation}
			\label{CondPower}
			c(z_k) \le \Delta(z_k) + \Sigma( z_k) \Psi^{-1}(\beta_k),
		\end{equation}
		the detector has conditional power at least $ 1- \beta_k $, $ 1 \le k \le K $, and then the constraint $ \sum_{k=1}^K \beta_k p_k \le \beta $  ensures a (marginal) power of at least $ 1-\beta$.  
		\item[(ii)] Suppose that the  threshold function $ c(\cdot) $ is selected  by 
		\[ c(z_k) = \Psi^{-1}(g_k/p_k) \bm 1_{\{p_k>0\}} + \Psi^{-1}(0) \bm 1_{\{p_k=0\}}, \qquad 1 \le k \le K, \] 
		for some nonnegative vector $ \vecg = (g_1, \ldots,g_K)^\top $. If  the $K+1$ inequalities
		\[
		0 \le g_k \le p_k \Psi( c_{\text{opt}}^*(z_k;\beta) ) , 1 \le k \le K, \ \sum_{k=1}^K g_k \ge 1-\alpha,
		\]
		where $ \ c_{\text{opt}}^*(z_k;\beta_k) = \Delta(z_k) + \Psi^{-1}(\beta) \Sigma(z_k)  $, $ 1 \le k \le K $,
		hold, then the detector operates on a significance level $ \alpha $, has detection power at least $ 1-\beta $ and associated average run length $ \frac{1}{1-\beta} $ under the alternative. If 
		\[
		g_k \le p_k \Psi( c_{\text{opt}}^*(z_k, \beta_k) ), 
		\]
		then the conditional power is at least $ 1 - \beta_k $, $ 1 \le k \le K $, and the constraint $ \sum_{k=1}^K p_k \beta_k \le \beta $ ensures a marginal power of at least $ 1 - \beta $.
	\end{itemize}
\end{theorem}

The result can be used to check whether a rule of interest has certain power under a given alternative.  However, note that, in general, the sufficient condition obtained in statement (ii) formulates a joint condition on the threshold function and the alternative parameterized by $ \Delta(\cdot) $ and $ \Sigma(\cdot) $. Specifically, for small deviations from the null hypothesis, there may be no solution of the set of sufficient inequalities: If the values $ c_{\text{opt}}^*(z_k) = \Delta(z_k) + \Sigma(z_k) \Psi^{-1}(\beta) $, $ 1 \le k \le K $, are positive but very small, then any selection of $ g_1, \ldots, g_K $ with $ 0 \le g_k < p_k $ and $ g_k \le p_k \Psi( c_{\text{opt}}^*(z_k) ) $,  $ 1 \le k \le K $,  may be too small to ensure the type I error rate constraint, which requires that the sum of the $ g_k $ is large enough to ensure $ \sum_{k=1}^K g_k \ge 1-\alpha $.

\subsection{Computing optimal solutions}

If the statistician is able to specify, in addition to $ \alpha $, all required conditional type II error rates, $  \beta_k $, or at least the required marginal type II error rate, $ \beta $, and has sufficient a priori knowledge to explicitly fix an alternative hypothesis in terms of $ \Delta( \cdot)  $ and $ \Sigma(\cdot) $, then the question arises how one can efficiently solve these inequalities. Admissible solutions can be found as follows: Determine (algorithmically) a level $ \alpha $ detector with conditional power $1-\beta_k $ given $z_k$ by finding an appropriate vector $\vecg $ satisfying the set of the inequalties 
\[
0 \le g_k \le p_k \Psi(\Delta(z_k) + \Sigma(z_k) \Psi^{-1}(\beta_k) ), 1 \le k \le K, \qquad \sum_{k=1}^K g_k \ge 1-\alpha. 
\]
These inequalities can be written in matrix form,
\begin{equation}
	\label{IneqMatrix}
	\matA \vecg \le \vecb, 
\end{equation}
if we define
\[
\matA = \begin{pmatrix} -\bm 1  \\ \mathbb{I}_K \end{pmatrix}, \qquad \vecb = \begin{pmatrix} \alpha-1 \\ \vecp  \# \Psi(\vecc_{\text{opt}}^*) \end{pmatrix},
\]
where $ \bm 1 = (1, \ldots, 1)^\top \in \R^K $, 
\[ \Psi(\vecc_{\text{opt}}^*) = ( \Psi(c_{\text{opt}}^*(z_1,\beta_1)), \ldots, \Psi(c_{\text{opt}}^*(z_K,\beta_K)))^\top \] 
and $ \veca \# \vecb = (a_i b_i)_{i=1}^n $, for $n$-vectors $ \veca = (a_1,\ldots, a_n)^\top $ and $ \vecb = (b_1, \ldots, b_n)^\top $, denotes pointwise multiplication. When aiming at a rule ensuring a marginal detection power of at least $ 1- \beta $, the $ c_{\text{opt}}^*(z_k, \beta_k) $ are replaced by $ c_{\text{opt}}^*(z_k,\beta )$.

Nonnegative solutions $ \vecg $ of \eqref{IneqMatrix} can be obtained by the $b$-rule algorithm, see \cite{AvisBohdan2004}, or by formulating the problem  as a linear optimization problem with inequality constraints and coefficients of the linear objective function set to zero. In general, there is no unique solution if the set of admissible solutions is nonempty. 

One can proceed and compute the optimal solution maximizing the conditional detection power of the smallest minority class. Thus, with $ k^* \in \{1, \ldots, K \} $ such that $ p_{k^*} = \min_{1 \le k \le K} p_k $, one solves the linear optimization problem under inequality constraints
\[
\min \vecd^\top \vecg \qquad \text{under the constraints $ \matA \vecg \le \vecb $.}
\]
where $ \vecd = (d_1, \ldots, d_K)^\top $ with $ d_{k^*} = 1 $ and $ d_i = 0 $ for $ i \not = k^* $.

\section{Nonparametric estimation and differentiability}
\label{Sec:PropThreshold}

Generalizing the setup discussed so far, consider the case of $q-$variate environmental information $ \vecZ_t $, $ q \in \N $, i.e., $ \calZ \subset \R^q $.

In applications, the distributions of $ U_t $ and $ \vecZ_t $ are often unknown, and thus any threshold function of class \eqref{GeneralClass} needs to be estimated, especially, if it depends on the law of $ Z_t $ as well. The case of a known c.d.f $ \Psi $ of $ U_t $ can be treated by replacing the unknown probabilities, $ p_k $, by their corresponding relative frequencies and has been studied in \cite{SteRafa2024}. Here, we consider the more involved case that $ \Psi $ is unknown as well. 

For simplicity of presentation, we focus on the  proportional rule $ c_{\text{prop}}( \cdot ) $ noting that a similar result can be obtained with minor changes for the $ \gamma $-proportional rule and the modified rule. A natural ansatz for estimation of $  c_{\text{prop}}( \cdot )  $, which keeps the structural form of the rule, is to replace the $p_k $ by relative frequencies again, and to nonparametrically estimate the c.d.f $ \Psi $ of the $ U_t $'s, $ t < t^*$,  by the empirical c.d.f (e.c.d.f) of the sample $ U_1, \ldots, U_n $. Let $ \hat{\vecp}_n  $ be the $K$-vector of relative frequencies, $ \hat p_k = \frac1n \sum_{t=1}^n \eins_{\{\vecZ_t=\vecz_k\}} $, of observing $ \vecz_k$ in the sample, $ 1 \le k \le K $. Define the estimator $ \hat{c}_{\text{prop}}( \cdot ) $ as the random map from the support $ \calZ $ of the $ \vecZ_t $'s to $ \R $ defined 
\[
\hat{c}_{\text{prop}}(z_k) = \hat \Psi_n^{-1}\left( \frac{1-\alpha}{\sum_{j=1}^K \hat p_j^2} \hat p_k  \right), \qquad 1 \le k \le K,
\]
where $  \hat \Psi_n^{-1}( \cdot ) $  is the sample quantile function (i.e., the left-continuous generalized inverse) of the marginal empirical distribution function $ \hat F_{U,n}(u) = \lim_{\vecz \to \bm\infty }\hat F_n(u, \vecz) $ associated to the e.c.d.f. 
\[
\hat F_n(u,\vecz) = \frac{1}{n} \sum_{t=1}^n \eins_{\{U_t \le u, \vecZ_t \le \vecz \}}, \qquad u \in \R, \vecz \in \R^q.
\]
If the $ U_t $ depend on $h$ past measurements as the CUSUM statistic discussed Section~\ref{Sec: Method}, one can consider  the e.c.d.f of  each $h$th pair and correct $n$ in all formulas of this section accordingly. For brevity and clarity of presentation, we consider the case of standardized raw measurements.

Since the asymptotic results of this section depend on the employed quantile estimator only through the weak limit of the empirical process, $ \sqrt{n} (\hat F_n(\cdot) -  F(\cdot) ) $, cf. Assumption (A4) below, one may also use other quantile estimators including smoothed estimators. For example, the quantile estimator based on the Bernstein-Durmeyer smoothing operator as proposed by and studied in \cite{PepRafSte2014} can be used, which satisfies (A4) under mild conditions.

\subsection{Preliminaries and assumptions} 

Denote for a function $ G $ on $ \R^q $ the (linear) difference operator\footnote{defined as $ \sum\limits_{\varepsilon_1, \ldots, \varepsilon_q \in \{0,1\}} (-1)^{q-\sum_{j=1}^q \varepsilon_j} G( b_1^{\varepsilon_1}a_1^{1-\varepsilon_1}, \ldots, b_q^{\varepsilon_q} a_q^{1-\varepsilon_q} )$}  $ D_{\veca}^{\vecb} G $ for $ \veca, \vecb \in \R^q $ with $ \veca \le \vecb $, such that $ D_{\veca}^{\vecb} G = \int_{(\veca,\vecb]} dG $ if $G$ is a c.d.f.. Weak convergence of random elements is understood in the sense of \cite{VaartWellner2023}, i.e. in the space $ (l^\infty, \| \cdot \|_\infty) $ of bounded functions,  and denoted by  $ \Rightarrow $.

We impose the following conditions.

\noindent
\textbf{Assumption (A1):} $ (U_t,\vecZ_t) $, $ t \ge 1 $, are independent and identically distributed with c.d.f. $ F $.

\noindent
\textbf{Assumption (A2):} The marginal c.d.f. $ \Psi = F_U $ is continuously differentiable on $ \text{supp}{(d\Psi)}^o $ with positive density and has no jumps at the boundary $ \partial \text{supp}(d\Psi)  $. 

\noindent
\textbf{Assumption (A3):} $ \vecp $ is given by $ p_k = D_{\veca_k}^{\vecb_k} F_\vecZ $, $ 1 \le k \le K $, for constants $ \veca_k < \vecb_k $, $ 1 \le k \le K $, and some continuous c.d.f. $ F_\vecZ $ on $ \R^q $. Further, $ \vecp $ is admissible in the sense that $ 0 < p_k < \frac{\| \vecp \|_2^2}{1-\alpha} $, $ 1 \le k \le K $.

\noindent
\textbf{Assumption (A4):} The empirical process associated to the sample $ \{ (U_t, \vecZ_t)  : 1 \le t \le n \} $ and the estimator $ \hat F_n $ satisfies
\[
\sqrt{n}( \hat F_n(\cdot) - F(\cdot ) ) \Rightarrow \calB^0_F(  \cdot ) 
\]
as $  n \to \infty $, for some $ F $-Brownian bridge  $ \calB^0_F $.

Assumption (A2) is mild for non-discrete measurements, $ X_t $, and is standard for results on the quantile process. It also allows us to use the calculus of Hadamard differentiability.  The first condition of Assumption (A3) holds for a latent variable approach where an unobservable (typically continuous) latent random vector $ \tilde \vecZ $ is assumed, and instead of $ \tilde \vecZ $ one observes the coarser information $ \vecz_k $ if $ \tilde\vecZ_t \in (\veca_k, \vecb_k] $. The second condition is needed to ensure that $ c_{\text{prop}}^*( \cdot ) $ is defined. 

Assumption (A4) holds for the  e.c.d.f., on which we confince ourselves in the sequel, assuming the observations, $ X_t$, are standardized with the true conditional mean and standard deviation. Extensions will be discussed later. 

To this end, recall that $ \calB^0(\vecx) $, $ \vecx \in \R^N $, is called a $F$-Brownian bridge process associated to a c.d.f. $ F$ defined on $ \R^N $, if $ \{ \calB^0( \vecx )  : \vecx \in \R^N \} $ is a Gaussian process with mean zero and covariance function
\[
\Cov( \calB^0(\vecx), \calB^0(\vecy) ) = F( \vecx \wedge \vecy ) - F(\vecx)F(\vecy), \qquad \vecx, \vecy \in \R^N.
\]
Here, $ \vecx \wedge \vecy = ( x_1 \wedge y_1, \ldots, x_N \wedge y_N )^\top $ with $ x \wedge y = \min(x,y) $ for $ x, y \in \R $. If $ F $ is the uniform distribution on the $N$-dimensional unit cube, we obtain the standard Brownian bridge, $ \calB_{st}^0 $, on $ [0,1]^N $ with covariance function $ \Cov( \calB^0_{st}(\vecu), \calB^0_{st}(\vecv) )= \prod_{i=1}^N (u_i \wedge v_i - u_i v_i) $, for $ \vecu = (u_1, \ldots, u_N)^\top, \vecv = (v_1, \ldots, v_N)^\top  \in \R^N $. A straightforward calculation shows that the covariance function of the empirical process $ \sqrt{n}( \hat F_n(\cdot) - F(\cdot ) )  $ coincides with the covariance of an $ F $-Brownian bridge.

Partition $ \vecx^\top = (\vecx_1^\top, \vecx_2^\top) $ with $ \vecx_1 \in \R^{N_1} $ and $ \vecx_2 \in \R^{N_2} $, $ N_1 + N_2 = N $.   Then the marginal processes with respect to $ \vecx_1 $ and $ \vecx_2 $, respectively, obtained by letting the other coordinates tend to infinity, are Brownian bridges as well. Specifically,
\[
\calB^0_1(\vecx) = \calB^0( \vecx_1, \infty \eins_{N_2}), \qquad \vecx \in \R^{N_1},
\]
is a $ F_1 $-Brownian bridge for the marginal c.d.f. $ F_1(\vecx) = F(\vecx, \infty \eins_{N-N_1} ) $, 
and
\[
\calB^0_2(\vecx) = \calB^0( \infty \eins_{N_1}, \vecx ), \qquad \vecx \in \R^{N_2},
\]
is a $F_2$-Brownian for the marginal c.d.f. $ F_2(\vecx) = F(\infty \eins_{N_1}, \vecx) $. 

%

\subsection{Differentiability and central limit theorem}

We approach the problem to derive a central limit theorem (CLT) for the proposed estimator $ \hat{c}_{n}( \cdot ) $ by representing it as a functional evaluated at the empirical measure $ \hat F_n $ of the sample and show that the functional is Hadamard  differentiable. The appropriate domain of distribution functions is
\[
\calF = \{ F : \text{$F$ is a c.d.f. such (A2) and (A3) hold true} \}.
\]
The elements of $ \calF $ serve as distribution functions for the pairs $(U,\vecZ) $ and $ (X,\vecZ) $, respectively, and we shall frequently denote them by $ F_{(U,\vecZ)} $ or $ F_{(X,\vecZ)} $ to clarify their role. The marginals with respect to $X$ and $U$ are denoted by $ F_X $ and $ F_U $, respectively, with the understanding that if $ F = F_{(X,\vecZ)} \in \calF $ specifies the c.d.f. of $ ( X, \vecZ) $, the notation $ F_U = \Psi $ stands for the c.d.f. of the standardized measurement $ U = (X-\mu(\vecZ))/ \sigma(\vecZ) $. Further, since by Assumption (A3) the law of $ \vecZ $ is determined by the continuous c.d.f. $ F_\vecZ $, i.e., the c.d.f. of the latent variable in a latent variable model, with some abuse of standard notation we identify the c.d.f. of $ \vecZ $ with that c.d.f..

The proposed estimator is induced by the real-valued functional defined by 
\[
T(F) = (T_1(F), \ldots, T_K(F))^\top , \quad T_k( F ) = F_U^{-1}\left( \frac{c_\alpha  D_{\veca_k}^{\vecb_k}F_\vecZ }{\sum_j (D_{\veca_j}^{\vecb_j} F_\vecZ)^2 } \right), \qquad F \in \calF,
\]
where $ D_{\veca_k}^{\vecb_k} F_\vecZ = dF_\vecZ(\veca_k, \vecb_k] \in (0,1) $, $ 1 \le k \le K $,   $ c_\alpha = 1-\alpha $, and $\veca_k, \vecb_k $ are the constants from Assumption (A3). Then
\[
	\hat{c}_{\text{prop}}(\vecz_k)  = T_k( \hat F_n ), \qquad 1 \le k \le K,
\]
and $ c_{\text{prop}}( \cdot) $ and $ \hat{c}_{\text{prop}}(\cdot) $ can be identified with $ T(F)  $ and $ T(\hat{F}_n) $, respectively.

Recall the definition of (directional) Hadamard differentiability, \cite{VaartWellner2023}: Let $ \calD, \calE $ be metrizable, topological vector spaces. A map $ \phi : \calF_\phi \to \calE $ from a domain $ \calF_\phi \subset \calD $ to $ \calE $ is called Hadamard differentiable at $ \theta \in \calF_\phi $ tangentially to $ \calD_{\calF,0} \subset \calD $, if there exists a continuous linear map $ \phi'_\theta : \calD_{\calF,0} \to \calE $ such that for all converging sequences $ 0 < t_n \to 0 $ and sequences $ \{ \Delta_n : n \ge 1 \} \subset \calD $ with $  \Delta_n \to \Delta $, $n \to \infty $, for some $ \Delta \in \calD_{\calF,0} $, and $ \theta + t_n \Delta_n \in \calF_\phi $ for all $n$,
\[
	\frac{\phi(\theta + t_n \Delta_n) - \phi(\theta)}{t_n}  \to \phi'_\theta( \Delta ), \qquad n \to \infty.
\]
The domain $ \calF_\phi \subset \calD $ of the map may be an arbitrary set, and it suffices that the derivative is defined on $ \calD_{\calF,0} $ only. Further, it suffices that $ \theta $ and all $ \Delta_n $, $ n \ge 1 $, are such that addition and multiplication are continuous operations; then it suffices that $ \calD_{\calF,0}  $ is a topological vector space. For example, in Skorohod spaces $ +$ and $ \cdot $ are continuous operations, if at least one operand is a continuous function. The latter is guaranteed if $ \calF_\phi $ is a set of continuous functions. 

The chain rule of Hadamard differentiability states that in the situation $  \calF_\phi \stackrel{\phi}{\to}  \calE \stackrel{\psi}{\to} \calG $ for maps $ \phi $ and $ \psi $, such that $ \phi $ is Hadamard differentiable at each $ \theta \in \calF_\phi $ tangentially to $ \calD_{\calF,0} $ and $ \psi $ is Hadamard differentiable tangentially to $ \phi'_\theta( \calD_{\calF,0} ) $,
the composition $ \psi \circ \phi : \calF_\phi \to \calG $ is Hadamard differentiable at $ \theta \in \calF_\phi $ tangentially to $ \calD_{\calF,0} $ with derivative $ (\psi \circ \phi)'_\theta : \calD_{\calF,0} \to \calG $ at $ \theta \in \calF_\phi $ given by $ (\psi \circ \phi)'_\theta(\Delta) = \psi'_{\phi(\theta)}( \phi'_\theta( \Delta ) )$, $ \Delta \in \calD_{\calF,0} $.

A slightly stronger notion is continuous Fr\'echet differentiability, \cite{Shao1993}: A map $ \phi : \calF_\phi \to \calE $ is called continuously Fr\'echet differentiable at $ \theta \in \calF_{\phi} $, if there exists some continuous and linear map $ T'_\theta : \calD \to \calE $, such that for all sequences $ \{ \theta_n \}, \{ \psi_n \} \subset \calF_\phi $ with $ \theta_n \to \theta $ and $ \psi_n \to \theta $, $ n \to \infty $, it holds
\begin{equation}
\label{FrechetDiff}
	\frac{T(\psi_n) - T(\theta_n) - T'_\theta( \psi_n - \theta_n )}{\rho(\psi_n, \theta_n ) } \to 0, \qquad n \to \infty.
\end{equation}
Here, $ \rho $ denotes the metric of $ \calF_\phi $.  One can replace $ \calD $ by the vector space $ \calD_{\calF}^{vec} $ (of linear combinations) induced by $ \calF_\phi $. 
$ \phi $ is  named continuously Fr\'echet differentiable at $ \theta \in \calF_{\phi} $ tangentially to $ \calD_{\calF,0} $, if there exists a linear map $ T'_\theta $ on $ \calD_{\calF}^{vec} $ (or even $ \calD$) which is continuous for all convergent sequences $ \{ \Delta_n \} \subset  \calD_{\calF}^{vec} $ (resp. $ \subset \calD$) with limit $ \Delta \in \calD_{\calF,0} $, and for all such sequences and all sequences $ \{ \theta_n \} \subset \calF_\phi $ with $ \theta_n \to \theta $ and  $ 0 < t_n \to 0 $  statement \eqref{FrechetDiff} holds with $ \psi_n = \theta_n + t_n \Delta_n $. 

\begin{lemma}
	\label{propHadamard}
	Let $ \rho $ be a translation invariant and absolutely homogenous metric. If both $ \calF $ and $ \calD $ are equipped with $ \rho $, then continuous Fr\'echet differentiability (tangentially to $ \calD_{\calF,0} $) implies Hadamard differentiability (tangentially to $ \calD_{\calF,0} $). 
\end{lemma}

For our present purposes, it suffices to consider the domain $ \calF_\phi = \calF $ as defined above (we indicate changes if required), i.e., $ \calF $ consists of all c.d.f.s such that (A2)-(A3) are satisfied. 

Anticipating that we shall later consider vector-valued measurements of dimension $q' \in \N $, we use the space of bounded functions,  $ \calE = \ell^\infty( \R^{q'+q} ) $, and 
\[ 
\calD_{\calF} = D( \R^{q'} \times \R^q ; \R ) \cap \{ f : \R^{q'} \times \R^{q} \to \R \text{ with } \lim_{|\vecx|_\infty\to \infty} f(\vecx) = 0  \}
\] 
as the vector space hosting directions, i.e.,  the  vector space of cadlag functions on $ \R^{q'} \times \R^q $, $ q $ the dimension of the $ \vecZ_t $ and $q'=1 $ for the present framework, vanishing at infinity. We  equip $ \calF_\psi $ and  $ \calD_{\calF}  $ with the Skorohod metric ensuring measurability, but will work with the stronger supnorm $ \| f \|_\infty $  or  the norm
$
	\| f \|_{\calF }= \| f \|_\infty + \| f \|_{L_1},
$
where
\[
\| f \|_{\infty}  = \sup_{\vecx \in \R^{q'}, \vecy \in \R^{q}} | f(\vecx,\vecy) |, \qquad
\| f \|_{L_1} = \int_{\R^{q'} \times \R^q} | f(\vecx) | \, d\vecx.
\]
The choice $ \| \cdot \|_{\calF } $ has been proposed by  \cite{Shao1993} and ensures that the mean functional is continuously Fr\'echet differentiable  and thus Hadamard differentiable as well. However, for  truncated moments the supnorm suffices.

The functional $ T_k $ is well defined for distribution functions $ F \in \calF $ and $ \calD_{\calF} $ hosts the trajectories of differences between empirical distribution functions and their expectation,  as well as trajectories of $F$-Brownian bridges. Since the latter are continuous, a suitable tangent space is 
\[ 
\calD_{\calF,0} = C(  \R^{q'} \times \R^q ; \R ) \cap \calD_\calF.
\] 
Note that, for clarity, we denote the proper direction space for domain $ \calF $ by $ \calD_\calF $ and the suitably chosen tangent space for $ \calD_\calF $ by $ \calD_{\calF,0} $, and we shall use this notation for other domains and their direction spaces arising in the proofs as well. For example, a proper direction space for the domain $ \calF \times (0,1) $ arising in the proofs is $ \calD_{\calF \times (0,1)} = \calD_\calF \times (-1,1) $, and both spaces can be equipped with the metric induced by $ (f, t) \mapsto  \| f \|_\calF + |t| $, $ f \in \calF $, $ t \in (-1,1) $.

Clearly, since $ \calZ $ is a finite set, it suffices to study the weak limit theory of $ \sqrt{n}( \hat{\vecc}_n - \vecc) $, where
\[
\hat{\vecc}_n = ( \hat{c}_n( \vecz_1), \ldots, \hat{c}_n( \vecz_K ))^\top, \qquad
\vecc = (c_{\text{prop}}(\vecz_1), \ldots, c_{\text{prop}}(\vecz_K) )^\top.
\]
The following theorem shows that the functional underlying $ \hat{\vecc}_n $ and $ \vecc $ is tangentially Hadamard differentiable such that the CLT follows from the functional delta method. 

\begin{theorem} 
	\label{ThHadamard}
	Under Assumptions (A2) and (A3) the functional $ T_k(F) $, $ F \in \calF $, is Hadamard differentiable at each $ F \in \calF $ tangentially to $ \calD_{\calF,0} $ with derivative
	\begin{align*}
	T'_{k,F}(\Delta) &= - \frac{\Delta_U \circ F_U^{-1}\left( \frac{c_\alpha D_{\veca_k}^{\vecb_k} F_\vecZ}{\sum_j (D_{\veca_j}^{\vecb_j} F_\vecZ)^2 } \right) }{f_U \left( F_U^{-1}\left( \frac{c_\alpha D_{\veca_k}^{\vecb_k} F_\vecZ}{\sum_j (D_{\veca_j}^{\vecb_j} F_\vecZ)^2 } \right) \right)} \\
	& \quad +  c_\alpha \frac{D_{\veca_k}^{\vecb_k} \Delta_\vecZ  \sum_j (D_{\veca_j}^{\vecb_j} F_\vecZ)^2 - 2 D_{\veca_k}^{\vecb_k} F_\vecZ \sum_j D_{\veca_j}^{\vecb_j} F_\vecZ D_{\veca_j}^{\vecb_j} \Delta_\vecZ  }{ f_U\left( F_U^{-1}\left(  \frac{c_\alpha D_{\veca_k}^{\vecb_k} F_\vecZ}{\sum_j (D_{\veca_j}^{\vecb_j} F_\vecZ)^2 }\right)\right)  \left( \sum_j (D_{\veca_j}^{\vecb_j} F_\vecZ)^2 \right)^2 } 
	\end{align*}
	for $ \Delta = (\Delta_X, \Delta_\vecZ) \in \calD_\calF $, where $ f_U = F_U' $.

	Under Assumptions (A1)-(A4) we have the weak convergence 
	\[
	\sqrt{n}( T_k( \hat{F}_n ) - T_k(F)) \Rightarrow T_{k,F}'(\calB^0_F ),  \qquad n \to \infty,
	\]
	where 
	\begin{align*}
	T'_{k,F}( \calB^0_F )  & - \frac{\calB_U^0 \circ F_U^{-1}\left( \frac{c_\alpha D_{\veca_k}^{\vecb_k} F_\vecZ}{\sum_j (D_{\veca_j}^{\vecb_j} F_\vecZ)^2 } \right) }{f_U \left( F_U^{-1}\left( \frac{c_\alpha D_{\veca_k}^{\vecb_k} F_\vecZ}{\sum_j (D_{\veca_j}^{\vecb_j} F_\vecZ)^2 } \right) \right)} \\
	& \quad +  c_\alpha \frac{ D_{\veca_k}^{\vecb_k} \calB_\vecZ^0   \sum_j (D_{\veca_j}^{\vecb_j} F_\vecZ )^2- 2 D_{\veca_k}^{\vecb_k} F_\vecZ \sum_j D_{\veca_j}^{\vecb_j} F_\vecZ D_{\veca_j}^{\vecb_j} \calB_\vecZ^0   }{ f_U\left( F_U^{-1}\left(  \frac{c_\alpha D_{\veca_k}^{\vecb_k} F_\vecZ}{\sum_j (D_{\veca_j}^{\vecb_j} F_\vecZ )^2}\right)\right)  \left( \sum_j (D_{\veca_j}^{\vecb_j} F_\vecZ)^2 \right)^2 } 
	\end{align*}
	and
	\[
	\sqrt{n}( \hat \vecc_n - \vecc )  = \sqrt{n} \left( T_k( \hat F_n ) - T_k (F) \right)_{k=1}^K \Rightarrow \left( T_{k,F}'(\calB^0_F )\right)_{k=1}^K,
	\]
	as $ n \to \infty $. Here, $ \calB^0_F $ is the $ F $-Brownian bridge from Assumption (A4) with marginal Brownian bridges $ \calB^0_U = \calB_{F_U}^0 $ (a $ F_U $-Brownian bridge on $ [0,  1]$) and $ \calB^0_\vecZ = \calB_{F_\vecZ}^0 $ (a $ F_\vecZ $-Brownian bridge on $ [0,1]^q $).
\end{theorem}

By the above theorem, the asymptotic distribution of $ \sqrt{n}( \hat{c}_{\text{prop}}( z_k) - c_{\text{prop}}(z_k)) $ is governed by the limiting Brownian bridge of the empirical process.

\section{Extensions to estimation from residuals}
\label{Sec:Residuals}


In applications, we need to estimate $ \mu(\cdot) $ and $ \sigma(\cdot) $ from the learning sample. 
When assuming that $ \mu(\cdot) = \mu $ and $ \sigma(\cdot) = \sigma $ are constants, one may use $ \hat\mu_n = \frac{1}{n} \sum_{t=1}^n X_t$, $ \hat \sigma_n^2 = \frac{1}{n} \sum_{t=1}^n X_t^2 - \hat\mu_n^2 $ and the residuals
\[
\hat{U}_t = \frac{X_t - \hat \mu_n}{ \hat \sigma_n}, \qquad 1 \le t \le n.
\]
We use  the residuals e.c.d.f.
\[
\tilde{F}_n(u, \vecz) = \frac{1}{n} \sum_{t=1}^n \eins( \hat{U}_t \le u, \vecZ_t \le \vecz )
\]
instead of $ \hat F_n(u,\vecz) $, and thus the associated estimator of $ c_{\text{prop}}( \vecz_k ) $ is defined by
\[
\tilde{c}_{\text{prop}}( \vecz_k ) = \tilde\Psi_n^{-1}\left(  \frac{1-\alpha}{\sum_{j=1}^K \hat p_j^2 } \hat p_k \right),
\]
where $ \tilde{\Psi}_n^{-1}(\cdot) $ is the quantile function of the first marginal of $ \tilde{F}_n(\cdot) $.  This leads to the thresholds
\[
\hat{t}_n(\vecz_k) = \hat \mu_n + \hat \sigma_n \tilde\Psi_n^{-1}\left(  \frac{1-\alpha}{\sum_{j=1}^K \hat p_j^2} \hat p_k \right)
\]
for the measurements $ X_t  $. 

Alternatively, one may wish to use estimators of conditional mean and standard deviation thus replacing $ \hat \mu_n $ by $ \hat \mu_n(\vecZ_i) $ and $ \hat{\sigma}_n $ by $ \hat{\sigma}_n(\vecZ_i) $ in the definition of $ \hat{U}_i $ which are then used in $ \tilde{F}_n(u,z) $ and the estimator $ \tilde{\Psi}_n^{-1}(\cdot) $. To keep the notation simple, we use the same notation for these quantities, since only the definition of the $ \hat{U}_i$ is affected. We propose to use category-wise (truncated) means and standard deviations, i.e.,
\begin{align*} 
	\hat{{\mu}}_n(\vecZ_i) &= \frac{ \sum_{t=1}^n X_t \eins_{\vecZ_t=\vecZ_i} }{ \sum_{t=1}^n \eins_{\vecZ_t=\vecZ_i} }, \\
	\hat{ {\sigma}}_n^{2}(\vecZ_i) &=  \frac{ \sum_{t=1}^n X_t^{2} \eins_{\vecZ_t=\vecZ_i} }{ \sum_{t=1}^n \eins_{\vecZ_t=\vecZ_i} } - \hat{{\mu}}_n^2(\vecZ_i),
\end{align*} 
such that the residuals are now defined by by $ \hat U_t = \frac{X_t - \hat\mu_n(\vecZ_t) }{\hat\sigma_n(\vecZ_t)} $, $ 1 \le t \le n $.
For an observation $ X_t $ with $ \vecZ_t= \vecz_k $, i.e., on the event $ \{\vecZ_t = \vecz_k \} $, this leads to the threshold
\[
\hat t_n( \vecZ_t ) = \hat \mu_n(\vecZ_t) + \hat{\sigma}_n(\vecZ_t) \tilde\Psi_n^{-1}\left(  \frac{1-\alpha}{\sum_{j=1}^K \hat p_j^2 } \hat p_k \right)
\]
with $ \hat \mu_n(\vecZ_t) = \hat \mu_n(\vecz_k) $ and $\hat{\sigma}_n(\vecZ_t)  = \hat{\sigma}_n(\vecz_k)  $.  

The rest of this section provides the details and shows that under mild assumptions the residual empirical process converges weakly to some Gaussian process, which in turn yields a CLT for the estimator $ \tilde{\vecc}_n $ when combined with our result on its differentiability. The residual empirical process has been studied from different perspectives in the literature to some extent. It is well known that estimation of parameters affects the asymptotic behaviour, see, e.g., \cite{Loynes1980} for residuals in regression models and the exposition in \cite[Ch.~5.5h.~]{ShorackWellner1986}. The results of this section contribute to this literature by studying the case of standardized multivariate observations by (truncated) moments and are tailored to our setting.

\subsection{Residual empirical process using marginal sample moments}

The following observation is crucial: The e.c.d.f. $ \tilde{F}_n(u, \vecz) $ can be expressed in terms of $ \hat{F}_{(X,\vecZ),n}(x, \vecz) $, since $ \hat{U}_t \le u $ is equivalent to $ X_t \le \hat{\mu}_n + \hat{\sigma}_n u $. Concretely, we have
\begin{equation}
	\label{Repr1}
	\tilde{F}_n(u, \vecz) = \hat{F}_{(X,\vecZ),n}( \hat{\mu}_n + \hat{\sigma}_n u, \vecz ), \qquad u \in \R, \vecz  \in \R^q.
\end{equation}
Establishing Hadamard differentiability of the underlying  functional  which maps $ F = F_{(X,\vecZ)} $ to the c.d.f. $ (u,\vecz) \mapsto  F( \mu(F_X) + \sigma(F_X) u, \vecz  ) $ is a special case of the corresponding 
result for the multivariate case, which is of interest for future generalizations. Therefore, we broaden the setting and assume for the rest of this section that $ X_t $ and hence $ U_t $ are random vectors of dimension $ q' \in \N $, which we indicate by using bold notation. Let us introduce the functionals
\[
\underline\bfmu(F) = \begin{pmatrix} \bfmu(F) \\ \vecnull_q \end{pmatrix},
\quad
\bfmu(F) = \int_{-\bftau}^{\bftau} \vecx d F_\vecX(\vecx),	
\]
and 
\[ 
\underline\bfSigma(F) = \begin{pmatrix}  \bfSigma(F) & \vecnull_q^\top\\ \vecnull_q & \matid_q \end{pmatrix},  \quad
\bfSigma(F) = \int_{-\bftau}^{\bftau}( \vecx - \bfmu(F_X))(\vecx-\bfmu(F_X))^\top \, d F_X(\vecx),
\]
for c.d.f.s $ F = F_{(X,\vecZ)} $, and denote the elements of $ \bfmu(F) $ by $ \mu_i(F) $ and those of $ \bfSigma(F) $ by  $ \Sigma_{ij}(F) $. As usual, we write $ \mu(F) $ and $ \sigma^2(F) $, if $ q' = 1 $.
In these definitions, $ \bftau = (\tau, \ldots, \tau)^\top \in \R^{q+q'} $ for a truncation constant $ \tau $ satisfying $ 0 <\tau < \infty $. By truncation, there is no need to assume finite moments.  Clearly, the integrals are understood elementwise, such that 
\[ 
\mu_i(F) = E(X_{1i} \eins(|X_{1i}|\le \tau)), \quad  \Sigma_{ij}(F) = E\left( (X_{1i}^{(\tau)}-\mu_i(F))(X_{1j}^{(\tau)}-\mu_j(F)) \right), 
\] 
for $1 \le i, j \le q' $, where $ X_{1i}^{(\tau)} = X_{1i} \eins(|X_{1i} |\le \tau ) $, $ 1\le i \le q' $. Note that the formula for $ \Sigma_{ij}(F)$ makes use of the fact that $ |\mu_i(F) | \le \tau $. If $ F(\cdot ) $ has bounded support, i.e. for bounded errors, and provided $ \tau $ is large enough, these functionals coincide with mean and covariance matrix. Otherwise, the bias will be small if $ \tau $ is selected large enough.
Define the  estimators  
\[ 
\hat{\underline{\bfmu}}_n = \begin{pmatrix} \bfmu(\hat F_{(X,\vecZ),n}) \\ \vecnull_q \end{pmatrix}, \qquad  \hat{\underline{\bfSigma}}_n = \begin{pmatrix}  \bfSigma(\hat F_{(X,Z),n}) & \vecnull_q^\top \\  \vecnull_q &  \matid_q \end{pmatrix}
\]
of $ \underline{\bfmu}(F) $ and $ \underline{\bfSigma}(F) $. Note that
\begin{align*}
\hat\bfmu_n &= \bfmu(\hat F_{(X,\vecZ),n}) = \frac{1}{n} \sum_{t=1}^n \vecX_t^{(\tau)}, \\
\quad 
\hat\bfSigma_n &= \bfSigma(\hat F_{(X,\vecZ),n})  = \frac{1}{n} \sum_{t=1}^n  (\vecX_t^{(\tau)} - \hat\bfmu_n)(\vecX_t^{(\tau)} - \hat\bfmu_n)^\top,
\end{align*}
where $ \vecX_t^{(\tau)} = ( X_{t1}^{(\tau)}, \ldots, X_{tq'}^{(\tau)})^\top $.

After these preparations, we are now in a position to discuss the crucial representations for the multivariate version of \eqref{Repr1}. First, note that 
\[
\vecU_t =\bfSigma^{-1/2}(F)(  \vecX_t  - \bfmu(F) ) \sim \Psi, \qquad 1 \le t \le n,
\]
and
\[
\hat\vecU_t = \hat\bfSigma_n^{-1/2}(\vecX_t  - \hat\bfmu_n ), \qquad 1 \le t \le n.
\]
The residuals e.c.d.f. 
\[
\hat F_{\hat \vecU,n}(\vecx ) = \frac1n \sum_{t=1}^n \eins_{\{ \hat \vecU_t \le \vecx  \}}, \qquad \vecx \in \R^{q'}, 
\]
of $ \hat\vecU_1, \ldots, \hat\vecU_n $ can then be written in the form
\[
\hat{F}_{\hat \vecU,n}( \vecx  ) = \hat{F}_{\vecX,n}\left( \hat{\bfmu}_n + \hat{\bfSigma}_n^{1/2} \vecx \right). 
\]
Noting that the affine transformation on the right side induces a map between spaces of distribution functions, the underlying  functional $ H : \calF \to \calE $ can be identified as the map
\[
H( F )(\vecx ) = F_\vecX( \bfmu(F) + \bfSigma^{1/2}(F) \vecx ), \qquad F \in \calF.
\]
Then we get the representations
\[
F_\vecU( \cdot ) = H( F_\vecX)(\cdot), \qquad \hat{F}_{\hat \vecU,n}( \cdot ) = H( \hat{F}_{\vecX,n} )( \cdot ), \qquad F \in \calF. 
\]
The  multivariate version of \eqref{Repr1} taking into account $ \vecZ $ is 
\begin{equation}
	\label{Repr2}
	\tilde{F}_n( \vecx, \vecz ) = \hat F_{(\vecX,\vecZ),n}\left( \underline{\hat \bfmu}_n + \underline{\hat\bfSigma}_n^{1/2} \begin{pmatrix} \vecx \\ \vecz \end{pmatrix}  \right),
\end{equation}
and if we introduce the functional $ \underline{H}: \calF \to \calE $,
\[
\underline{H}( F )(\cdot) = F( \underline{\bfmu}(F) + \underline{\bfSigma}^{1/2}(F) \cdot), \qquad F \in  \calF,
\]
we have the  representations
\begin{align}
	\tilde F_n(\cdot ) &= \underline{H}( \hat F_{(\vecX,\vecZ),n} )( \cdot ) ,\\
	F_{(\vecU,\vecZ)}( \cdot) & = \underline{H}( F_{(\vecX,\vecZ)} )( \cdot ) .
\end{align}
Below we shall show that $ \underline{H} $ is Hadamard differentiable from which the intended CLT and bootstrap CLT will follow.

Let
\[
\calF_2 = \calF \cap \calS_2, \calD_{\calF_2} = \calD_\calF, \calD_{\calF_2,0} = \calD_{\calF,0} , 
\]
where
\[
\qquad \calS_2 = \{ F: \text{$\bfSigma(F)$ is regular} \}. 
\]
For $ F \in \calF' $ the results of the previous Lemma are applicable and the function $ \underline{H} $ turns out to be Hadamard differentiable. Denote by $ A(\vecm, \matS) $ the affine transformation 
\[
\begin{pmatrix} \vecx \\ \vecz \end{pmatrix} \mapsto 
\begin{pmatrix}  \vecm + \matS \vecx \\ \vecz \end{pmatrix}, \qquad 
\begin{pmatrix} \vecx \\ \vecz \end{pmatrix} \in \R^{q'} \times \R^q,
\]

\begin{theorem} 
	\label{ThResidualProcess}
	Suppose that Assumptions (A1)--(A4) are fulfilled. Then the functional $ \underline{H} : \calF_2 \to \calE $,
	\[
	F \mapsto \underline{H}(F)(\vecx, \vecz) = F  \left( \underline{\bfmu}(F) + \underline{\bfSigma}(F)^{1/2} \left( \vecx \atop  \vecz \right) \right), 
	\]
	is Hadamard-differentiable at each $ F \in \calF_2 $ with derivative
	\[
	\underline{H}'_F(\Delta ) = A\left( \int_{-\bftau}^{\bftau}  \vecu \, d\Delta( \vecu ),  D \bfSigma(F)^{1/2} \vecop  \bfSigma'_F( \Delta ) \right),
	\]
	for  $ \Delta \in \calD_{\calF_2,0} $, where $ D \bfSigma(F)^{1/2}  $ and $ \bfSigma'_F( \cdot) $ are defined in \eqref{DefDifferentialSigmaRoot} and \eqref{DefDifferentialS}, respectively. 
		
	The residual empirical process converges weakly,
	\[
	\sqrt{n}( \hat{F}_{\hat \vecU,n}(\cdot) - F_{\vecU}( \cdot) ) \Rightarrow \calB_{\bfmu(F),\bfSigma(F)}^0(\cdot) =
	\underline{H}'_F( \calB_F^0 )(\cdot),
	\]
	as $ n \to \infty $. 
	
\end{theorem}

We are now in a position to provide the weak convergence of the estimator $  \tilde{\vecc}_n $  using the quantile function associated to the e.c.d.f. of the residuals, which follows from  Theorem~\ref{ThResidualProcess}.

\begin{theorem}
	\label{Th_CLT}
	\[
	\sqrt{n}( \tilde \vecc_n - \vecc )   \Rightarrow ( T'_{k, \underline{H}(F)}( \calB^0_{\bfmu(F),\bfSigma(F)}) )_{k=1}^K ,
	\]
	as $ n \to \infty $.
\end{theorem}

\subsection{Empirical residual process using conditional  sample moments}

Let us now assume that the truncated conditional mean and covariance matrix depend on $ \vecz$, such that
\[
\bfmu(\vecz) = \EE( \vecX^{(\tau)} | \vecZ = \vecz), \qquad \bfSigma(\vecz) = \Cov( \vecX^{(\tau)} | \vecZ=\vecz ).
\]  
Natural (truncated) estimators are obtained by plugging in $ \hat F_{(\vecX,\vecZ),n} $ instead of $ F = F_{(\vecX,\vecZ)}$. Thus, let for $ \vecz \in \calZ = \{ \vecz_1, \ldots, \vecz_K \} $
\[
\hat A_n(\vecz) = A(\hat F_n, \vecz ) = \frac1n \sum_{t=1}^n \vecX_t^{(\tau)} \eins_{\{\vecZ_i=\vecz\}}, \qquad
\hat N_n(\vecz) = N( \hat F_n, \vecz ) = \frac1n \sum_{t=1}^n \eins_{\{\vecZ_t=\vecz\}},
\]
and
\[
\hat \vecB_n(\vecz) = \vecB( \hat F_n, \vecz ) = \frac{1}{n} \sum_{t=1}^n (\vecX_t^{(\tau)} - \mu(\hat F_n,\vecz))(\vecX_t^{(\tau)} - \mu(\hat F_n,\vecz))^\top \eins_{\{\vecZ_t=\vecz\}}.
\]
The resulting estimators are defined for any $ \vecz \in \calZ $ with $ \hat{N}_n(\vecz) > 0 $ by
\[
\hat \bfmu_n(\vecz) = \hat A_n(\vecz) / \hat N_n(\vecz)
\]
and
\[
\hat \bfSigma_n(\vecz) = \hat \vecB_n(\vecz) / \hat N_n(\vecz).
\]
If we define for any c.d.f. $ F = F_{(\vecX,\vecZ)} \in \calF_2 $ and $ \vecz \in \calZ $ 
\begin{align*}
	A(F,\vecz) &= \int_{-\bftau}^{\bftau} \vecx \eins_{\{ \vecu=\vecz \}} \, dF(\vecx,\vecu), \\
	N(F,\vecz) & = \int_{\veca_k}^{\vecb_k} \, dF_\vecZ(\vecu),  \quad \text{if $\vecz = \vecz_k$},\\
	\vecB(F,\vecz) &= \int_{-\bftau}^{\bftau} (\vecx - \bfmu(F,\vecz))(\vecx - \bfmu(F,\vecz) )^\top \eins_{\{\vecu=\vecz\}}\, d F(\vecx, \vecu),  
\end{align*}
then  
\[ 
\bfmu(F,\vecz) = A(F,\vecz) / N(F,\vecz) , \qquad 
\bfSigma(F,\vecz) = \vecB(F,\vecz) / N(F,\vecz).
\]
The Hadamard differentiability of $ A(F,\vecz) $ and $ N(F,\vecz) $ follows similarly as for the truncated moments, and an application of the quotient rule of Hadamard differentiability entails the Hadamard differentiability of $ \bfmu(F,\vecz) $ and $ \bfSigma(F,\vecz) $, which again yields the differentiability of
\[
\underline{H}(F,\vecz) = \underline{H}(F,\vecz)(\vecx, \vecz) = F  \left( \underline{\bfmu}(F,\vecz) + \underline{\bfSigma}(F,\vecz)^{1/2} \left( \vecx \atop  \vecz \right) \right),
\]
and in turn the CLT of $ \tilde{c}_n(\vecz) $. For sake of brevity of presentation, we omit the details and formulas.

\section{Bootstrap}
\label{Sec:Bootstrap}

A feasible and usually accurate general approach for uncertainty quantification of estimators is the bootstrap. In our setting, we need to approximate the stochastic behaviour of the residual empirical process from which the estimator $ \tilde{c}_n(\cdot) $ is calculated. This problem has been addressed in the literature for linear model residuals, \cite{KoulLahiri1994}, as well as for residuals in nonparametric models, see \cite{NeumeyervanKeilegom2019} and the references given there. The specific result of the last section, which allows to reduce the problem to the empirical process of the original data, provides us with the following bootstrap approach when combined with the Hadamard differentiability of the estimator of interest.

Given the learning sample $ (X_1, \vecZ_1), \ldots, (X_n, \vecZ_n) $, let 
\[ (X_1^*, \vecZ_1^*), \ldots, (X_n^*, \vecZ_n^*)  \stackrel{i.i.d.}{\sim} \hat F_n \] be a nonparametric bootstrap sample of conditionally i.i.d. observations distributed according to the e.c.d.f. $ \hat F_n $. Next, if the estimator $ \tilde{c}_{\text{prop}}(\cdot) $ has been computed using marginal standardization, calculate the bootstrap residuals
\[
\hat U_t^* = \frac{U_t^* - \hat{{\mu}}_n^*}{ \hat{ {\sigma}}_n^*}, \qquad 1 \le i \le n,
\]
where  
\begin{align*} 
	\hat{{\mu}}_n^* &= \mu( \hat F_{(X,\vecZ), n}^* ) = \frac{1}{n} \sum_{i=1}^n X_i^{(\tau)*}, \\
	\hat{ {\sigma}}_n^{*2} &= \sigma(\hat F_{(X,\vecZ),n}^{*2} ) = \frac{1}{n} \sum_{i=1}^n (X_i^{(\tau)*}  - \overline{X^{(\tau)*}})^2.
\end{align*}
Otherwise, if conditional sample mean and conditional sample standard deviation have been used, calculate the bootstrap residuals
\[
\hat U_t^* = \frac{U_t^* - \hat{{\mu}}_n^*(\vecZ_t^*)}{ \hat{ {\sigma}}_n^*(\vecZ_t^*)}, \qquad 1 \le t \le n,
\]
where
\begin{align*} 
	\hat{{\mu}}_n^*(\vecZ_t^*) &= \frac{ \sum_{j=1}^n X_j^{(\tau)*} \eins_{\vecZ_j^*=\vecZ_t^*} }{ \sum_{j=1}^n \eins_{\vecZ_j^*=\vecZ_t^*} }, \\
	\hat{ {\sigma}}_n^{*2}(\vecZ_t^*) &=  \frac{ \sum_{j=1}^n X_j^{(\tau)*2} \eins_{\vecZ_j^*=\vecZ_t^*} }{ \sum_{j=1}^n \eins_{\vecZ_j^*=\vecZ_t^*} } - \hat{{\mu}}_n^*(\vecZ_t^*)^2 .
\end{align*}
Let $ \tilde F_n^* $ be the e.c.d.f. of $ (\hat U_1^*, \vecZ_1^*), \ldots, (\hat U_n^*, \vecZ_n^*) $. The bootstrap estimator is now defined by
\[
\hat{c}_{\text{prop}}^*( z_k ) = \tilde\Psi_n^*{}^{-1} \left( \frac{1-\alpha}{\sum_{j=1}^K (\hat p_j^*)^2 } \hat p_k^* \right)
\]
where $ \hat p_j^* $, $ 1 \le j \le K $, are the relative frequencies in the bootstrap sample $ \vecZ_1^*, \ldots, \vecZ_n^* $ and
$
\tilde \Psi_n^*{}^{-1}( \cdot )
$ is the quantile function of the first marginal d.f. of $ \tilde{F}_n^*(\cdot ) $.

\begin{theorem} 
	\label{Th_BTCLT}
	Assume that (A1)-(A4) are fulfilled. Then, under the conditional probability $ P^* $ given $ (X_1,\vecZ_1), \ldots, (X_n,\vecZ_n) $, we have outer $ P $-almost surely
	\[
	\left( \sqrt{n}( \tilde c_{\text{prop}}^*(\vecz_k ) - \tilde c_{\text{prop}}( \vecz_k ) ) \right)_{k=1}^K
	\Rightarrow ( T'_{k, \underline{H}(F)}( \calB^0_{\bfmu(F),\bfSigma(F)}) )_{k=1}^K,
	\]
	as $ n \to \infty $, such that the bootstrap is outer $P$-almost surely consistent.
\end{theorem}

\section{Example}
\label{Sec: Example}

To illustrate the approach, we analyze the Pima Indians diabetes dataset available from kaggle. It consists of $768$ observations from female Pima Indians (Akimel O'odham) and includes medical measurements relevant for the prediction and diagnosis of diabetes mellitus as well as the true diagnosis (diabetic/non-diabetic). The dataset is selected to study how diabetes can be predicted. Each observation corresponds to a woman with a non-diabetes glucose level, i.e. glucose 2h after a glucose tolerance test below $200$ mg/dl, at the examination, who either developed diabetes within the next five years or a  glucose-tolerance-test five or more years later failed to revealed diabetes mellitus. Measurements from a single examination are provided, \cite{PimaIndians1988}.  It is well established that the body mass index (BMI) is a severe risk factor for diabetes mellitus. Therefore, for purposes of illustration, we set up thresholds for the observed plasma glucose concentration, $X_t$,  adapted for the dichotomized BMI, $ Z_t $, in order to generate predictions and screen the population. The dataset allows to evaluate this approach, since the true diagnosis (diabetes within the next five years) is available. Keeping only observations with positive values, we are left with $ n = 752 $ observations, see Table~\ref{SummaryStats} for summary statistics. Although there is no gold standard for the design and evaluation of a screening test,  correctly detecting almost all true positives (diabetes cases), i.e. high sensitivity, as well as ensuring a sufficiently  small type I error of false positives among the true negatives (healthy woman), i.e. high specificity within the healthy population, can be seen as reasonable criteria. 

A $10\% $-high-risk class (high-risk BMI, HBMI) was defined by $ \{ Z_t > \hat q_{0.9}(Z) \} $, where $ \hat q_{0.9}(Z) $ stands for the sample $ 0.9 $-quantile of BMI, i.e., by the top decile, and the remaining cases represent the no-high-risk class. \color{black} Let us first discuss optimality properties for a significance level $ \alpha = 10\% $ and type II error rates $ \beta_1 = 1\% $ and $ \beta_2 = 2\%$ for $ \Delta = (1,4)  $ and $ \Sigma = (1,1) $. Replacing unkowns by their estimates, we get $ c_{\text{opt}}^*(z_1) = 1.674 $ and $ c_{\text{opt}}^*(z_2) = 1.946 $,  so that the upper bounds for candidate $ g_k$'s  are $ b_1 = \hat p_1 \Psi( c_{\text{opt}}^*(z_1) ) = 0.093$ and $ b_2 = \hat p_2 \Psi( c_{\text{opt}}^*(z_2) ) = 0.8511 $. Solving the problem to minimize $ g_1 $ under the constraints $ \matA \vecb $ yields the solution $ \hat g_1^* = 0.0489  $, $ \hat g_2^* = 0.851 $ with corresponding thresholds $ 128 $ and $ 179 $. \color{black}
The estimated proportional rule uses $ \hat g_1 = 0.011 $ and $ \hat g_2 = 0.889$. Since $ \hat g_1 \le b_1 = 0.093 $ and $ \hat g_2  \le b_2 = 0.8511$, for the specified alternative the proportional rule can be expected to have conditional type II error rates of $ 1\%$ and $2\% $, respecitively, and a marginal detection power of ca. $ 98.1\% $. The associated thresholds and resulting summary statistics are given in Table~\ref{MainTab2} discussed in greater detail below. 

Based on the residuals estimator using estimated conditional mean and standard deviation, the thresholds are estimated by $ \hat c_{n,\text{HBMI}} = -1.108 $ for the high-risk minority class and $ \hat c_{n,\text{No-HBMI}} = 2.462 $ for the majority class. Associated thresholds for the raw observations are given in the Table~\ref{MainTab2}. The bootstrap estimates of their standard deviations based on $ B = 5000 $ Monte-Carlo replicates  are $ 0.032 $ and $0.083 $, respectively. 

\subsection{Comparison and evaluation of monitoring rules}

Table~\ref{MainTab2} summarizes alarm rates, true positive and true negative rates when classifying the observations of the learning sample using the proportional rule, the optimal rule and, as comparisons, the constant rule and logistic regression. All  estimates given there are in-sample estimates. Below we report about a (bootstrap) simulation study to further evaluate the rules. For the proportional rule and the optimal rule $ 2  \times 2 $ contingency tables of  true positives (TP), false negatives (FN), false posities (FP) and true negatives (TN) are shown for both classes, see Table~\ref{ContingTab}, which allows the reader to calculate further evaluation metrics.

The proportional rule has an appealing high true positive rate (TPR, sensitivity) for high-risk cases (HBMI) and correctly identifies $ 97.9\%$ of all those women which will be diagnosed diabetes within the next five years. The true negative rate (TNR, specificity) for high-risk cases, is $ 36\% $. 
For the no high-risk class (No-HBMI) the rule correctly identifies practically almost all healthy women, as the specificity is $ 100\% $. The low sensitivity ($3.2\%$) among the no high-risk cases is, however, expected, since the rule is designed to achieve an overall false alarm rate of at most $ \alpha = 10\% $, and - by design - we have increased the sensitivity for the high-risk class at the cost of the sensitivity for the majority class.  


We also constructed an optimal threshold using the alternative given by $ \Delta(z_1) = \Delta(z_2) = 4 $, $ \Sigma(z_1) = \Sigma(z_2) = 1 $, for conditional type II error rates $ \beta_1 = 1\% $ and $ \beta_2 = 2\% $. For this specification, the resulting threshold for the high risk HBMI class is higher leading to an alarm rate of $ 56.6\% $ for that class and $ 5.8\% $ for the no high-risk cases, yielding an overall alarm rate of $ 10.9\% $. Sensitivity and specificity are both $ 75 \% $ for the high-risk HBMI class and $ 15\%$ and $ 98\% $ for the  No-HBMI class. 

	\begin{table}[ht]
		\centering\begin{tabular}{lrrrrrrr}  \hline & N & Mean & SD & min & max & prev. \\ 
			\hline 
			HBMI& 76 			& 135 & 31 & 67 & 199 & $0.63$  \\   
			No-HBMI & 676 & 121 & 30 & 44 & 198 &$0.32 $\\
			Overall & 752 & 122 & 31 & 44 & 199 & $0.35 $\\    \hline\end{tabular}
		\caption{Summary statistics of the plasma glucose level and (future) prevalence (prev.) of diabetes within the next five years after examination.}
		\label{SummaryStats}
	\end{table}
	
	\begin{table}[ht]
		\begin{small}
			\centering\begin{tabular}{|r|rrrr|rrrr|}  \hline 
				& \multicolumn{4}{c|}{Proportional Rule} & \multicolumn{4}{c|}{Constant Rule} 			\\
				& Thr & Alarm & TPR & TNR         & Thr & Alarm & TPR & TNR  \\   \hline 
				HBMI       & 100 & 0.855 & 0.979 & 0.36    & 180 & 0.092 & 0.15 & 1.00  \\   
				No-HBMI & 195 & 0.012 & 0.032 & 1.00   & 165 & 0.102  & 0.26 & 0.97 \\
				Overall &  			&  0.097 & 0.277 & 0.85 &           &  0.101 & 0.29 & 0.84 \\    \hline 
				& \multicolumn{4}{c|}{Optimal Rule} & \multicolumn{4}{c|}{Logistic Regression} \\
				HBMI  & 128  &    0.566  &     0.75    &    0.75 & -& 0.14 & 0.23 & 0.86 \\
				No-HBMI & 179     & 0.058   &    0.15       & 0.98 & -&	0.08 & 0.25 & 0.75 \\  
				Overall &    &    0.109   &    0.31   &    0.83&	 & 0.09 & 0.25 & 0.75   \\    \hline 
			\end{tabular}
			\caption{Thresholds, alarm rates, true positive rates (TPR, sensitivity) and true negative rates (TNR, specificity) for the high-risk class of high body mass index (HBMI) cases and the no high-risk class (No-HBMI).}
			\label{MainTab2}
		\end{small}
	\end{table}


		A comparison with the constant rule is insightful. Here, the threshold takes into account the estimates for level and dispersion of each class and thus uses class-specific thresholds, but it does not distribute sensitivity across classes. The summary in Table~\ref{MainTab2} shows that the constant alarm threshold is too high for the high-risk HBMI cases to yield a high true positive rate. Especially, the contingency table, see Table~\ref{ContingTab2}, reveals that there are only $7$ alarms in this class. As a consequence, the sensitivity (TPR) is only $ 15\% $, whereas the specificity (TNR) is $100\%$.  
		For the majority class of cases without high-risk (No-HBMI) the sensitivity is higher than for the proportional rule, but it nevertheless reaches only $25\% $. Similar as for the proportional rule, the specificity for this class is very high, $75\%$. 
		Note that the constant rule suffers from the common phenomenon that the minority class is not handled properly as it detects risk cases only rarely. A plausibe explanation based on the summary statistics is that these cases show higher glucose values  years before diabetes is eventually diagnosed, namely $ 135 $ on average compared to $ 121$ on average for the no high-risk HBMI class. This  leads to a large threshold, such that future diabetes cases can only be predicted with low probability. That phenomenon clearly shows that the conditional standardization of the classical approach is not sufficient, and  it is likely that it is present in other applications as well. The proportional rule effectively mitigates this effect by its inherent weighting mechanism. 
		
		\begin{table}[ht]\centering\begin{tabular}{rrrr}  \hline Diabetes & Alarm & No alarm & $\sum$ \\   \hline Yes  &  47 &   1 &  48 \\  
				No  &  18 &  10 &  28 \\ 
				$\sum$ &  65 &  11 &  76 \\    \hline\end{tabular} \qquad
			\centering\begin{tabular}{rrrr}  \hline Diabetes & Alarm & No alarm & $\sum$ \\  
				\hline Yes &  7 & 209 & 216 \\ 
				No &   1 & 459 & 460 \\ 
				$\sum$ &  8 & 668 & 676 \\    \hline\end{tabular}
			\caption{Proportional rule: Contingency tables for high BMI  class (HBMI, left) and no high BMI (No-HBMI, right).}
			\label{ContingTab}
		\end{table}

		\begin{table}[ht]\centering\begin{tabular}{rrrr}  \hline Diabetes & Alarm & No alarm & $\sum$ \\   \hline
				Yes &   7 &  41 &  48 \\  
				No  &   0 &  28 &  28 \\  
				$\sum$ &   7 &  69 &  76 \\    \hline\end{tabular} \qquad
			\begin{tabular}{rrrr}  \hline Diabetes & Alarm & No alarm & $\sum$ \\   \hline 
				Yes &  57 & 159 & 216 \\  
				No  &   12 & 448 & 460 \\  
				$\sum$ &  69 & 607 & 676 \\    \hline\end{tabular}
			\caption{Constant rule: Contingency tables for high BMI risk class (HBMI, left) and no high BMI No-HBMI, right).}
			\label{ContingTab2}
		\end{table}
		
		Lastly, we compared the results with a logistic regression, although such a comparison is difficult: A logistic regression fits a model for the probability to develop diabetes as a logit function of the regressors, and the resulting classification does not control the type I error rate of falsely deciding in favor of future diabetes. We used glucose level and an indicator of the event $ Z_t > \hat{q}_{0.9} $ as regressors. As common practice, an observation was classified as '1' if the predicted value exceeds $ 0.5 $, i.e. in this case an alarm was raised. The main characteristics are shown in Table~\ref{MainTab2}. For the data at hand, this rule maintains the type I error rate and has favorable true negative rates. However, it also suffers from low sensitivity to detect true positives among risk cases, which is only $ 21\% $ compared to $ 98\% $ for the proportional rule. 
		
		\subsection{Simulation assessment of population screening}
		
		To further evaluate the rules, the following simulation analysis was conducted, in order to get some insights into the accuracy when applying the approach to screen a large sample of a  population.  In step 1, a bootstrap sample of size $n$ from the data was drawn to estimate the thresholds. Then, in step 2 of the procedure, a further sample of size $N = 10,000$ was drawn, using a smooth bootstrap by convolving each resampled measurement with Gaussian noise with standard deviation $ 1.59 s N^{-1/5} $ where $s$ stands for the estimated standard deviation. This sample represents the population sample to be screened. By using a smoothed bootstrap we circumvent the issue that for a nonparametric resampling many individual observations occur multiple times in the drawn sample. Step 3 consists in  comparing the  glucose levels of these $N$ observations with the thresholds calculated in the first step and determining the associated false alarm rates. Averaging these estimates over $B = 5,000 $ replications of the above three steps yields a marginal false alarm rate of $ 9.715 \% $ and conditional false alarm rates of $ 85.187\%$ for the high-risk BMI class (No-HBMI) and $ 1.329\%$ for the no high-risk class. The simulation was also carried out for the optimal test. Here the conditional alarm rates are $ 56.39 \% $, $ 5.66\% $ and $ 10.73\% $. 
		These results suggest that the proprosed proportional rule  using estimated thresholds as well as the optimal one works reliable in practice.

		\appendix
		
		\section{Proofs}
		\label{Sec: Proofs}
		
		\subsection{Modified rule by dichotomizing categories}
		
If $ p_{i_0} := \min_j p_j \ge \alpha $, one may take category $ i_0 $ and fuse the remaining categories. If all probabilities $p_j $ are larger than $ \alpha$, one can proceed as follows: 
Assume that the categories are ordered such that $  p_1 \le \cdots \le p_K $. Let us fix some $ 1 \le k_0 < K $ and assign to the smallest $ k_0 $ categories a small probability $ p_{\min} $, and to the remaining large categories a probability $ p_{\max} >p_{\min}$, such that the condition
\begin{equation}
	\label{CondProbWeights}
	\frac{p_{\min}}{p_{\max}} \ge \frac{p_2+ \cdots + p_{k_0}}{p_1+ \cdots + p_{k_0}}
\end{equation}
holds. \color{black} Here, $ p_{\min} $ and $ p_{\max} $ need to be probabilities, but they serve as scores assigned to the small and large categories, respectively, and thus may differ from the true probabilities and may also be functions of $ p_1, \ldots, p_k $. \color{black} Specifically, for any given constant $ p_{\min} $ or smooth function of $ \vecp $ with $ p_{\min} \ge \alpha $, one may choose $ p_{\max} = p_{\min} \frac{p_1+ \cdots + p_{k_0}}{p_2+ \cdots + p_{k_0}} $ as a smooth function of $ \vecp $.  Now define
\[
\tilde{p}_k =  p_{\min} \varphi(k_0-k) + p_{\max} (1-\varphi)(k_0-k)
\]
and
\begin{equation}
	\label{ProportionalRuleMod}
	c_{\text{mod}}(z_k) = \Psi^{-1}\left( \frac{(1-\alpha) \tilde p_k}{\sum_{j=1}^K p_j \tilde p_j} \right), \qquad 1 \le k \le K.
\end{equation}
Here, $ \varphi $ is a $C^2$-smoothed version of the indicator $ \eins_{[0,\infty)} $ with $ \varphi(x) =  \eins_{[0,\infty)} $ for $ |x| \ge 1 $. 

\begin{lemma}
	\label{LemmaProportionalRuleMod}
	If $ p_1 \le \alpha $, $ \vecp > 0 $ (element-wise) and $ 0 < p_{\min} < p_{\max} < 1 $ are smooth functions of $ \vecp $ satisfying \eqref{CondProbWeights}, then $ c_{\text{mod}}(\cdot) $ is defined.
\end{lemma}

		\subsection{Proofs of Section~\ref{Sec: Method}}

		\begin{proof}[Proof of Lemma~\ref{LemmaProportionalRuleMod}] We need to show that the argument of $ \Psi^{-1} $ in the definition of $ c_{\text{mod}}(\cdot) $ is less than $1$. Note that  $ \sum_j p_j \tilde p_j = p_{\min}(p_1 + \cdots + p_{k_0}) + p_{\max}(p_{k_0+1}+ \cdots + p_K) $. Since $ 1- \alpha \le p_2 + \cdots + p_K $, the numerator of the argument can be bounded by
			\[
			(1-\alpha) \tilde p_k \le \left\{ \begin{array}{cc} (p_2 + \cdots + p_K) p_{\min}, & \qquad k \le k_0, \\
				(p_2 + \cdots + p_K) p_{\max} & \qquad k > k_0 \end{array} \right..
			\]
			Clearly, $ \tilde p_k / \sum_j p_j \tilde p_j < 1$, if $ k \le k _0 $, and since \eqref{CondProbWeights} implies $  p_{\min}( p_1+ \cdots + p_{k_0} ) \ge p_{\max} (p_2+ \cdots + p_{k_0}) $, we obtain for $ k > k_0 $ 
			\begin{align*}
				\frac{\tilde p_k}{\sum_j p_j \tilde p_j} & \le \frac{(p_2 + \cdots + p_K) p_{\max}}{p_{\min}(p_1 + \cdots + p_{k_0}) + p_{\max}(p_{k_0+1}+ \cdots + p_K)}  \\
				& \le \frac{(p_2 + \cdots + p_K) p_{\max}}{p_{\max} (p_2+ \cdots + p_{k_0}) + p_{\max}(p_{k_0+1}+ \cdots + p_K)} \\
				& = p_2 + \cdots + p_K \le  1,
			\end{align*}
			\color{black} such that $ \frac{(1-\alpha)\tilde p_k}{\sum_j p_j \tilde p_j} < 1 $. \color{black}
		\end{proof}
		
		\subsection{Proofs of Section~\ref{Sec: Optimality}}

		\begin{proof}[Proof of Theorem~\ref{ThOptimality}]
			Under model \eqref{H1Model} it holds
			\[
			\frac{X_t- \mu(Z_t)}{\Sigma(Z_t)\sigma(Z_t)} - \frac{\Delta(z)}{\Sigma(z)} \sim \Psi,
			\]
			and therefore the detection probability to decide in favor of $ H_1 $, $p_d $, can be calculated as 
			\begin{align*}
				p_d & = P_1\left( U_t > c(Z_t) \right) \\
				&= \sum_{k=1}^K P\left( \frac{X_t- \mu(Z_t)}{\Sigma(z_k)\sigma(Z_t)} - \frac{\Delta(z_k)}{\Sigma(z_k)} > \frac{c(z_k)-\Delta(z_k)}{\Sigma(z_k)}  \right) p_k \\
				& = \sum_{k=1}^K \left[1-\Psi\left(  \frac{c(z_k)-\Delta(z_k)}{\Sigma(z_k)} \right) \right] p_k,
			\end{align*}
			where the term in brackets is the conditional power given $ Z_t = z_k $. 
			Consequently, the procedure attains at most a type II error rate $ \beta \in (0,1) $ and thus a detection power of at least $ 1-\beta $ to detect the alternative $ H_1 $, if and only if
			\begin{equation}
				\label{H1IneqDiscr}
				\sum_{k=1}^K \Psi\left(  \frac{c(z_k)-\Delta(z_k)}{\Sigma(z_k)}\right) p_k \le \beta.
			\end{equation}
			By monotonicity, any function $ c(\cdot) $ such that 
			$ \frac{c(z_k)-\Delta(z_k)}{\Sigma(z_k)} \le \Psi^{-1}(\beta) $ solves \eqref{H1IneqDiscr}. Further, any function $ c( \cdot) $ with $ \frac{c(z_k)-\Delta(z_k)}{\Sigma(z_k)} \le \Psi^{-1}(\beta_k) $ has conditional type II error rate at most $ \beta_k $, and then $ \sum_{k=1}^K p_k \beta_k \le \beta $ ensures a marginal type II errror rate of at most $ \beta $.  
			This shows (i). Assertion (ii) follows by combining (i) and statement (iv) of  \cite[Theorem~1]{SteRafa2024}, cf. \eqref{GeneralClass}, where $ c(z_k) =\Psi^{-1}( g_k/p_k ) $ for all $ p_k > 0 $ and $ c(z_k) = 0 $ if $ p_k = 0$. Then \eqref{H1IneqDiscr} holds, if
			\[ g_k \le p_k \Psi( \Delta(z_k) + \Psi^{-1}(\beta) \Sigma(z_k) ) = p_k \Psi( c_{\text{opt}}^*(z_k,\beta) ). \]  
			We can conclude that a sufficient condition to satisfy the type I and type II error rate constraints is
			as follows: For all $ 1 \le k \le K $ with $ p_k > 0 $
			\[
			0 < g_k < p_k, \qquad  \text{as well as} \qquad g_k \le p_k  \Psi( c_{\text{opt}}^*(z_k,\beta)  ), 
			\] 	 
			and $g_k = 0 $ for all $1 \le k \le K $ with $ p_k = 0 $. Noting that $ 0 \le g_k \le p_k  \Psi( c_{\text{opt}}^*(z_k)  ) $ implies $ g_k < p_k$  when $ p_k > 0 $  and $ g_k = 0 $ if $ p_k = 0$ establishes the result. The claims concerning the conditional detection power follow using similar arguments as in (i).
		\end{proof}

		\subsection{Proofs of Section~\ref{Sec:PropThreshold}}
		
		\begin{proof}[Proof of Lemma~\ref{propHadamard}] Fix $ \theta \in \calF_\psi $, let $ \{ \Delta, \Delta_n \} \subset \calD_{\calF} $  (resp. $\calD_{\calF}^{vec} $)  be a convergent sequence with limit $ \Delta \in \calD_{\calF,0}$ and let $ 0 <t_n \to 0 $ be an arbitrary sequence of positive reals. Put $ \psi_n = \theta + t_n \Delta_n $ and $ \theta_n = \theta $. Note that, by definition of $\psi_n  $, we have $ \rho( \psi_n - \theta, 0 ) = \rho( t_n \Delta_n, 0)  = t_n \rho( \Delta_n, 0) $, since $ \rho $ is absolutely homogenous, so that $ t_n = \frac{\rho(\psi_n - \theta)}{\rho(\Delta_n,0)}  $. Moreover, using the triangle inequality and the assumption that $ \calF_\psi $ as well as $ \calD $ are equipped with the metric $ \rho $,
			\[ 
			\rho(\psi_n, \theta ) = \rho( \theta + t_n \Delta_n, \theta) \le t_n \rho( \Delta_n, 0 ) 
			\le t_n ( \rho(\Delta, 0) + \rho( \Delta_n, \Delta) ) \to 0,   
			\]
			as $ n \to \infty $. Further, since $ T'_\theta $ is continuous and linear on  $ \calD_\calF $  (resp. linear on $\calD_{\calF}^{vec} $ and continuous for sequences in $ \calD_\calF^{vec} $ with limit in $ \calD_{\calF,0} $)  and thus bounded, we have $  T'_\theta( \Delta ) = T'_\theta(\Delta_n) + o(1) $. By continuous Fr\'echet differentiability of $ T $ at $ \theta $ (tangentially to $ \calD_{\calF,0} $)  with Fr\'echet derivative $ T'_F(\cdot) $, we now obtain
			\begin{align*}
				\frac{T(\theta+t_n \Delta_n) - T(\theta)}{t_n} - T'_\theta( \Delta ) 
				& = \frac{T(\psi_n) - T(\theta_n) - t_n T'_\theta( \Delta_n)}{\rho(\psi_n - \theta,0)} \rho(\Delta_n,0) + o(1) \\
				& = \frac{T(\psi_n) - T(\theta_n) - T'_\theta( t_n \Delta_n)}{\rho(\psi_n - \theta,0)} \rho(\Delta_n,0) + o(1)  \\
				& = \frac{T(\psi_n) - T(\theta_n) - T'_\theta( \psi_n - \theta_n) }{\rho(\psi_n, \theta)} \rho(\Delta_n,0) + o(1) \\
				& = o(1),
			\end{align*}
			if $ n \to \infty $, where the second equation uses the linearity of $ T'_\theta $ and the third one $ \theta = \theta_ n $, $ \rho(\Delta_n,0) = O(1) $, the translation invariance of $ \rho $ and the definition of continuous Fr\'echet  differentiability.
		\end{proof}

		\begin{proof}[Proof of Theorem~\ref{ThHadamard}]
			We decompose the functional $ T_k(F) $
			\[
			T_k(F) = \Omega \circ ( F, S ( F  )) = \Omega \circ R( F ),
			\]
			where 
			\[ \Omega : \mathcal{F} \times (0,1) \to \R, \qquad  \Omega( F, p ) = F_U^{-1}( p ), \qquad  (F,p) \in \mathcal{F} \times (0,1), 
			\] 
			$ R := (\id, S) : \calF \to \calF \times \R $ is defined by
			\[ 
			R( F ) = ( F, S( F ) ), \qquad  F \in \calF, 
			\]
			and 
			\[
			S: \calF \to (0,1), \qquad  S( F ) = \frac{c_\alpha D_{\veca_k}^{\vecb_k} F_\vecZ}{\sum_j (D_{\veca_j}^{\vecb_j} F_\vecZ)^2 }, \qquad  F  \in \calF.
			\]
			Note that Assumption (A3) ensures that $S$ maps to $(0,1)$. 
			The claim follows, if we show that each of these functionals is Hadamard differentiable and that the assumptions of the chain rule are satisfied.   
			
			1) Recall that the difference operator $ D_{\veca}^{\vecb} g $ is defined for any function $g$ on $ \R^q $ and is linear in $ g $. Since numerator and denominator of $ S(F) $ are Hadamard differentiable and the denominator is positive, $S(F) $ is Hadamard differentiable with derivative $ S' : \mathcal{D}_{\calF,0} \to \R $  at  $ F \in \calF$ given by 
			\[ 
			S'_{F}(\Delta) = c_\alpha \frac{D_{\veca_k}^{\vecb_k} \Delta_\vecZ  \sum_j (D_{\veca_j}^{\vecb_j} F_\vecZ)^2 - 2 D_{\veca_k}^{\vecb_k} F_\vecZ \sum_j D_{\veca_j}^{\vecb_j} F_\vecZ D_{\veca_j}^{\vecb_j} \Delta_\vecZ  }{ \left( \sum_j (D_{\veca_j}^{\vecb_j} F_\vecZ)^2 \right)^2 } 
			\] 
			for $ \Delta = (\Delta_U,\Delta_\vecZ) \in \mathcal{D}_{\calF} $. Note that $ S'_{F}(\Delta) $ does not depend on $ \Delta_U $.
			
			2) Since the coordinate maps of $R$ are Hadamard differentiable, $R$ inherits this property and has derivative at $ G \in \calF $ 
			\[
			R_G'(\Delta) = (\Delta, S'_{G}(\Delta) ).
			\]
			for $ \Delta \in \calD_{\calF,0} $. 
			
			3) Further, by definition of $ \Omega(F,p) = F_U^{-1}(p)$ as the inversion of the 1st marginal c.d.f., $F_U $, of $F = F_{(U,\vecZ)} $ evaluated at $ p$, we have 
			\[
			\Omega'_{(F,p)}( \Delta ) = \begin{pmatrix}  \frac{\partial F_U^{-1}(p)}{ \partial F } , & \frac{\partial F_U^{-1}(p)}{ \partial p} \end{pmatrix} 
			\begin{pmatrix} (\Delta_U,\Delta_\vecZ) \\ \Delta p \end{pmatrix} =
			- \frac{\Delta_U \circ F_U^{-1}(p)}{f_U ( F_U^{-1}( p ) )} + \frac{\Delta p}{f_U( F_U^{-1}(p))  }  
			\]
			for $ \Delta = ((\Delta_U, \Delta_\vecZ), \Delta p) \in \mathcal{D}_{\calF \times [-1,1],0} $.
			
			4) Lasty, apply the chain rule to the maps $ \Omega : \calF \times [0,1] \to \R $ and  $R = (\id,S)  : \calF \to \calF \times \R $.  We obtain for the derivative, $ T'_{k,F} $, of $ T_k $ at $ F \in \calF $ 
			\begin{align*}
				T'_{k,F}(\Delta ) &=  (\Omega \circ R)'_{F}( \Delta ) \\
				&= \Omega'_{R(F )}(  R'_{F}(\Delta) )   \\
				&= 
				\Omega'_{(F, S(F))}( (\Delta, S'_{F}( \Delta )  ) ) \\
				& = - \frac{(\Delta_U \circ F_U^{-1}( S(F)))} {f_U ( F_U^{-1}( S(F) )} + \frac{S'_{F}( \Delta )}{f_U( F_U^{-1}( S(F) )  } \\
				& = - \frac{\Delta_U \circ F_U^{-1}\left( \frac{c_\alpha D_k F_\vecZ}{\sum_j (D_j F_\vecZ)^2 } \right) }{f_U \left( F_U^{-1}\left( \frac{c_\alpha D_k  F_\vecZ}{\sum_j (D_j F_\vecZ)^2 } \right) \right)} +  c_\alpha \frac{D_k  \Delta_\vecZ  \sum_j (D_j F_\vecZ)^2 - 2 D_k  F_\vecZ \sum_j D_j  F_\vecZ D_j  \Delta_\vecZ  }{ f_U\left( F_U^{-1}\left(  \frac{c_\alpha D_k F_\vecZ}{\sum_j (D_j  F_\vecZ)^2 }\right)\right)  \left( \sum_j (D_j F_\vecZ)^2 \right)^2 } 
			\end{align*}
			for any direction $ \Delta = (\Delta_U, \Delta_\vecZ) \in \calD_{\calF,0} $,
			where $ D_k F_\vecZ = D_{\veca_k}^{\vecb_k} F_\vecZ $. 
			
			The second assertion now follows easily.  By Assumption (A4), the empirical process 
			$ \sqrt{n}( \hat F_n - F ) $ converges weakly to the $F$-Brownian bridge $ \calB_F^0 $ which has marginal Brownian bridges $ \calB^0_U $ and $ \calB^0_\vecZ $, such that 
			\[
			\sqrt{n}( \hat{F}_{U,n}  - F_U) \Rightarrow \calB_U^0, \qquad 
			\sqrt{n}( \hat{F}_{\vecZ,n}  - F_\vecZ) \Rightarrow \calB_\vecZ^0,
			\]
			as $ n \to \infty $. Therefore, the functional delta method for Hadamard differentiable functionals yields 
			\[
			\sqrt{n}( T_k(\hat F_n) - T_k(F) ) \Rightarrow T'_{k,F}( \calB^0_F ) 
			\]	
			where $ T'_{k,F}( \calB^0_F )  $ equals
			\[
			   - \frac{\calB_U^0 \circ F_U^{-1}\left( \frac{c_\alpha D_{\veca_k}^{\vecb_k} F_\vecZ}{\sum_j (D_{\veca_j}^{\vecb_j} F_\vecZ)^2 } \right) }{f_U \left( F_U^{-1}\left( \frac{c_\alpha D_{\veca_k}^{\vecb_k} F_\vecZ}{\sum_j (D_{\veca_j}^{\vecb_j} F_\vecZ)^2 } \right) \right)} +  c_\alpha \frac{ D_{\veca_k}^{\vecb_k} \calB_\vecZ^0   \sum_j (D_{\veca_j}^{\vecb_j} F_\vecZ )^2- 2 D_{\veca_k}^{\vecb_k} F_\vecZ \sum_j D_{\veca_j}^{\vecb_j} F_\vecZ D_{\veca_j}^{\vecb_j} \calB_\vecZ^0   }{ f_U\left( F_U^{-1}\left(  \frac{c_\alpha D_{\veca_k}^{\vecb_k} F_\vecZ}{\sum_j (D_{\veca_j}^{\vecb_j} F_\vecZ )^2}\right)\right)  \left( \sum_j (D_{\veca_j}^{\vecb_j} F_\vecZ)^2 \right)^2 }. 
			\]
			Now, the second statement follows  form the Cramer-Wold device and the above results.
		\end{proof}

		\subsection{Proofs of Section~\ref{Sec:Residuals}}
		
		The proof of Theorem~\ref{ThResidualProcess} uses the following auxiliary results from matrix calculus.
		
		\begin{lemma}
			\label{Hilfslemma}
			(i) Let $ \matA, \matB $ be square $N$-dimensional matrices. Then
			\begin{align}
				\label{F1}
				\vecop \matA \matB &= (\matid_N \otimes \matA ) \vecop \matB, \\
				\label{F2}
				\vecop \matB \matA &=  (\matA^\top \otimes \matid_N) \vecop \matB.
			\end{align}
			(ii) Let $ \matC $ be regular square $N$-dimensional matrix. Then a solution of the equation
			\[
			\matC \matD + \matD  \matC = \matid_N
			\]
			is given by
			\[
			\matD = \vecop^{-1}\left(  (\matid_N \otimes \matC + \matC^\top \otimes \matid_N )^{-1} \vecop \matid_N\right),
			\]
			where $ \vecop^{-1} : \R^{N^2} \to \R^{N \times N}$ denotes the inverse vectorization.
		\end{lemma}
		
		\begin{proof}
			Let $ \matB = (\vecb_1, \ldots, \vecb_N) $, where $ \vecb_i $ are the columns of $ \matB $, so that $ \matA \matB $ has columns $ \matA\vecb_i $. Thus, $ \vecop \matA \matB $ is the vector that stacks $ \matA \vecb_1, \ldots, \matA \vecb_N $, and it follows that
			\[
			(\matid_N \otimes \matA ) \vecop \matB = \begin{pmatrix} \matA & & \\ & \ddots & \\ & & \matA  \end{pmatrix} \begin{pmatrix} \vecb_1 \\\vdots \\ \vecb_N \end{pmatrix} = \vecop (\matA \matB).
			\]
			Further, since $  \matA^\top \otimes \matid_N  = ( a_{ji} \matid_N)_{1 \le i \le N \atop 1 \le j \le N} $ is a block-matrix with square $N\times N$-block $ a_{ji} \matid_N $ in row $i$ and column $j$, and since 
			$ \vecop \matB $ is a $ N \times 1 $-block matrix with blocks $ \vecb_i $ in row $i$, the $i$th element of $ \matA^\top \otimes \matid_N  \vecop \matB $ is given by
			$  (a_{1i} \matid_N, \ldots, a_{Ni} \matid_N  ) \vecop \matB =  a_{1i} \vecb_1+a_{2i} \vecb_2 + \cdots + a_{Ni} \vecb_N $. The latter coincides with the $i$th block of $ \vecop \matB \matA $, which is given by  $ \matB \veca_i = \sum_{k=1}^N \vecb_k a_{ik} $ (the linear combination of the columns of $ \matB $ with coefficients $ a_{1i}, \ldots, a_{Ni} $). Hence \eqref{F2} follows. 
			
			To establish the second assertion, observe that by \eqref{F1} and \eqref{F2}
			\begin{align*} 
				\matC \matD + \matD  \matC = \matid_N & \Leftrightarrow \vecop \matC \matD + \vecop \matD \matC = \vecop \matid_N \\
				& \Leftrightarrow (\matid_N \otimes \matC + \matC^\top \otimes \matid_N )\vecop \matD = \vecop \matid_N.
			\end{align*} 
			Solving the last equation for $\vecop \matD $ gives 
			$ \vecop \matD =  (\matid_N \otimes \matC + \matC^\top \otimes \matid_N )^{-1}  \vecop \matid_N$. Applying $ \vecop^{-1} $ establishes the result.
		\end{proof}
		
		Let us start by studying the differentiability of the functionals $ \bfmu(F) $ and $ \bfSigma(F) $. This requires to elaborate some differentials of matrix-valued functions on matrix spaces. Let us recall the following notions and facts, \cite{MagnusNeudecker1999}. A $ \R^{m\times p}$-valued function $ f(\matX) $ from a domain $ A \subset \R^{n\times q'} $ is differentiable at $ \matX_0  \in A $ if for all small $ \bfDelta \in \R^{n \times q'} $ such that $ \matX_0 + \bfDelta \in A $ there exists a $ mp \times nq' $ matrix $ f'_{\matX_0} $ such that
		\begin{equation}
			\label{LinApprox}
			\vecop f(\matX) - \vecop f( \matX_0 ) = f'_{\matX_0} \vecop \bfDelta + o( \| \bfDelta \| ), \qquad \bfDelta \to \matnull.
		\end{equation}
		Then $ f'_{\matX_0} $ is called differential of $f$ at $ \matX_0$ and also denoted $ D f(\matX_0) $. Note that  $ D f(\matX_0) $ is the usual Jacobian matrix at $ \matX_0 $ of the vectorized function $ f( \vecop \matX )  = \vecop f(\matX) $ with respect to the variables $ \vecop \matX$. When applying the chain rule, one composes differentials and tends to write formal expressions such as $ D h(g(\matX_0)) Dg(\matX_0) $. Here, it is important to note that the differential $ Df(\matX_0) $, a matrix of dimension $ mp \times nq'$,  operates on the $ n \times q' $ matrix $ \bfDelta $ via the vectorization operator in view of \eqref{LinApprox}, and we shall make this explicit in our notation. 
		
		In view of the definition of $ \underline{H} $ we need to handle differentials of patterned functions where only the upper left block has non-vanishing derivatives. Thus, consider the embedding  $ \mathcal{E} : \R^{n \times q'} \to \R^{N \times Q} $, 
		\[
		\mathcal{e}  \matX = \begin{pmatrix} \matX & \vecnull_{Q-q'}^\top \\ \vecnull_{N-n} & \vecnull_{(N-n)\times (Q-q')} \end{pmatrix}
		\] 
		of a $ n \times q' $ dimensional matrix $ \matX $ into  $ \R^{N \times Q} $.  Consider the differentiable function 
		\[ 
		\tilde f: \mathcal{e} A \to \R^{m \times p}, \qquad \tilde f( \tilde \matX ) = f( \matX ),  \tilde \matX \in \R^{N\times Q},
		\] 
		for some differentiable function $ f: A \to \R^{m \times p} $, $ A \subset \R^{n\times q'}$, (as above), where  $ \matX $ is the upper left $ N \times Q $ submatrix of $ \tilde \matX $. Then $ \tilde f( \mathcal{e} \matX ) = f( \matX ) $ for all $ \matX \in \R^{n \times q'} $,
		all partial derivatives of $ \tilde{f} $ with respect to variables $ X_{ij} $, $ 1 \le i \le n,  1 \le j \le q' $, coincide with those of $ f $, and the partial derivatives with respect to $ X_{ij} $, $ i > n $ or $ j >q' $, vanish.
		Therefore, the differential $ D \tilde f( \tilde \matX_0 ) \in \R^{mp \times NQ} $, i.e. the Jacobian at $ \tilde \matX_0 \in \R^{N\times Q }$ of the vectorized function $ f( \vecop \tilde \matX ) $, $ \tilde \matX \in \R^{N \times Q} $, is completely determined by $ D f( \vecop \matX_0 ) \in \R^{mp \times nq'}$, where $ \matX_0 $ is the upper left $ n \times q' $ submatrix of $ \matX_0 $,
		and there exists a unique map $ \imath : \R^{mp \times nq'} \to \R^{mp \times N Q} $,  such that 
		\[
		\tilde f'_{\tilde \matX_0} = D \tilde f( \tilde \matX_0 ) = \imath D f( \matX_0 ). 
		\]
		Then for all small $ \bfDelta \in \R^{N \times Q} $ 
		\[
		\vecop \tilde f( \tilde \matX ) - \vecop \tilde f( \tilde \matX_0  ) 
		=  \imath D f( \matX_0 ) \vecop \bfDelta + o( \| \bfDelta \| ).
		\]
		In particular, in this situation the linear approximation of $ f( \matX) $ is simply lifted to the embedding Euclidean space, and the relevant information is given by $ D f(\matX_0) \in \R^{mp  \times nq'} $ and directions $ \bfDelta \in \R^{n \times q'} $. 
		
		\begin{lemma}
			\label{LemmaDerivs}
			(i) $ \bfmu(F) $ and $\underline{\bfmu}(F) $ are Hadamard differentiable at $ F \in \calF $ tangentially to $ \calD_{\calF,0}$ with derivatives
			\[
			\bfmu'_F(\Delta) = \int_{-\bftau}^{\bftau} u \, d \Delta(u), \qquad \Delta \in \calD_{\calF,0}.
			\]
			and $ \underline{\bfmu}'_F(\Delta) = ( \bfmu'_F(\Delta)^\top, \vecnull_q^\top)^\top$. \\
			(ii) $ \bfSigma(F) $ is Hadamard differentiable at $ F \in \calF $ tangentially to $ \calD_{\calF,0} $ with derivatives
			\[
			\bfSigma'_F( \Delta ) = \int_{-\bftau}^{\bftau} u u^\top \, d\Delta(u) - \int_{-\bftau}^{\bftau}  u \, d \Delta(u)  \bfmu(F) ^\top - \bfmu(F) \int_{-\bftau}^{\bftau} u^\top \, d \Delta(u), \qquad \Delta \in \calD_{\calF,0},
			\]
			(iii) The matrix differential of $ \bfSigma(F) \mapsto \bfSigma(F)^{1/2} $, ensuring that 
			\[ 
				\vecop (\bfSigma(F) + \Delta )^{1/2} - \vecop \bfSigma(F)^{1/2} = D \bfSigma(F)^{1/2} \vecop \Delta + o(\| \Delta \|),
			\]
			is given by
			\begin{equation}
				\label{DefDifferentialSigmaRoot}
				D \bfSigma(F)^{1/2} =  \vecop^{-1}\left(  (\matid_n \otimes \matA^{1/2} + \matA^{\top/2}  \otimes \matid_n )^{-1} \vecop \matid_N \right),
			\end{equation}
			The map $ F \mapsto \bfSigma(F)^{1/2} $, $ F \in \calF_2 $, is Hadamard differentiable at $ F \in \calF_2 $ tangentially to $ \calD_{\calF_2}  $ with derivative
			\begin{equation}
				\label{DefDifferentialS}
				(\bfSigma^{1/2})'_F( \Delta ) = D \bfSigma(F)^{1/2} \vecop \bfSigma'_F( \Delta ), \qquad \Delta \in \calD_{\calF}'
			\end{equation}
			(iv) If $ \bfSigma(F) > 0 $, then the matrix differential of $ \bfSigma(F) \mapsto \bfSigma(F)^{-1/2} $, ensuring that 
			\[ \vecop (\bfSigma(F) + \Delta )^{-1/2} - \vecop \bfSigma(F)^{-1/2} = D \bfSigma(F)^{-1/2} \vecop \Delta + o(\| \Delta \|),
			\]
			is given by
			\begin{equation}
				\label{DefDifferentialSigmaRootInverse}
				D \bfSigma(F)^{-1/2} = - (\bfSigma(F)^{-1/2} \otimes \bfSigma(F)^{-1/2} ) \vecop^{-1}\left(  ( \matid_n \otimes \matA^{1/2} + \matA^{\top/2}  \otimes \matid_n )^{-1} \vecop \matid_N \right).
			\end{equation}
			The map $ F \mapsto \bfSigma(F)^{-1/2} $, $ F \in \calF_2 $, is Hadamard differentiable at $ F \in \calF_2 $ tangentially to $ \calD_{\calF_2,0}  $ with derivative
			\begin{equation}
				\label{DefDifferentialInverseRootS}
				(\bfSigma^{-1/2})'_F( \Delta ) = D \bfSigma(F)^{-1/2} \vecop \bfSigma'_F( \Delta ), \qquad \Delta \in \calD_{\calF_2,0}.
			\end{equation}
			(v) The derivatives of  $ \underline{\bfSigma}(F) $,  and $ \underline{\bfSigma}(F)^{1/2} $, and $ \underline{\bfSigma}(F)^{-1/2} $ are given by embedding those of $ \bfSigma(F), \bfSigma(F)^{1/2} $ and $ \bfSigma(F)^{-1/2} $ using the map $ \imath $. This means, these functionals are Hadamard differentiable with derivatives given by
			\begin{align*}
				\underline{\bfSigma}'_F( \Delta ) &= \imath \bfSigma'_F( \Delta ), \\
				(\underline{\bfSigma}^{1/2})'_F( \Delta ) &= \imath D \bfSigma(F)^{1/2} \vecop \imath \bfSigma'_F( \Delta ), \\
				(\underline{\bfSigma}^{-1/2})'_F( \Delta ) &= \imath D \bfSigma(F)^{-1/2} \vecop \imath \bfSigma'_F( \Delta ).
			\end{align*} 
		\end{lemma}

		\begin{proof}[Proof of Lemma~\ref{LemmaDerivs}] For (i) see \cite{VaartWellner2023}. To prove (ii) note that $ \bfSigma(F) = \int_{-\tau}^\tau  \vecu  \vecu^\top \, d F(\vecu) - \bfmu(F) \bfmu(F)^\top $.  Clearly, the Hadamard derivative of the first term at $ F \in \calF_2 $ is given by $ \int_{-\tau}^\tau  \vecu \vecu^\top \, d \Delta $ for any direction $ \Delta \in \calD_{\calF_2,0}' $. Indeed, for each pair $i,j \in \{1, \ldots, q \} $, w.l.o.g. $ i < j $, the functions $ (u_i,u_j) \mapsto u_i u_j  \eins_{[-a,a]} $ have bounded total variation in the sense of Hardy and Krause and the functions $ (u_i,u_j) \mapsto \Delta_{ij}(u_i, u_j) = \Delta(u_1, \ldots, u_i, \ldots, u_j, \ldots, u_q) $ are continuous, such that the integral $ \int_{-\tau}^\tau u_i u_j \, d \Delta_{ij}(u_i, u_j)  $  exists in the Stieltjes sense via integration by parts, see \cite{Steland2025} for a comprehensive discussion. Further, for $ t > 0 $ and $ \Delta \in \calD_{\calF_2,0} $
			\begin{align*}
				& \frac{\bfmu(F+t\Delta)\bfmu(F+t\Delta)^\top  - \bfmu(F) \bfmu(F)^\top}{t} \\
				& \qquad = \frac{[\bfmu(F+t\Delta)-\bfmu(F)] \bfmu(F+t\Delta)^\top + \bfmu(F)[\bfmu(F+t\Delta)-\bfmu(F)]^\top}{t} \\
				& \qquad \to \mu'_F(\Delta) \bfmu(F)^\top + \bfmu(F) \mu'_F(\Delta)^\top
			\end{align*}
			as $ t \to 0 $, which shows the assertion in view of (i).
			
			To show (iii) note that for a symmetric square matrix $ \matA $ we have the differential 
			\[ D \matA^{1/2} = ( \matid_n \otimes \matA^{1/2} + \matA^{\top/2}  \otimes \matid_n )^{-1},
			\] 
			i.e.,
			\[
			\vecop (\matA + \bfDelta )^{1/2} - \vecop \matA^{1/2} =  ( \matid_n \otimes \matA^{1/2} + \matA^{\top/2}  \otimes \matid_n )^{-1} \vecop \bfDelta + o( \| \bfDelta \| ).
			\]
			This result is less known and thus deserves a  proof. Clearly, $ \matA^{1/2} \matA^{1/2} = \matA $, and the chain rule for the matrix derivative leads to the equation
			\[
			D \matA^{1/2} \matA^{1/2} + \matA^{1/2} D \matA^{1/2} = D \matA = \matid_{n^2}.
			\]
			Next, by Lemma~\ref{Hilfslemma}~(ii) with $ \matC = D \matA^{1/2} $ and $ \matD = \matA^{1/2}$,
			\[
			D \matA^{1/2} = \vecop^{-1}\left(  (\matid_n \otimes \matA^{1/2} + \matA^{\top/2}  \otimes \matid_n )^{-1} \vecop \matid_N \right).
			\]	
			Now, the second assertion follows easily from the above results by the chain rule of Hadamard differentiability applied to the composition $ F \mapsto \bfSigma(F) \mapsto \bfSigma^{1/2}(F) $.
			
			The proof of (iv) goes along the same lines. It is well known that the differential of matrix inversion of a symmetric matrix is given by $ D \matA^{-1} = - (\matA^{-1} \otimes \matA^{-1} ) $, i.e., 
			\[
			\vecop (\matA + \bfDelta )^{-1} - \vecop \matA^{-1} = - (\matA^{-1} \otimes \matA^{-1} )  \vecop \bfDelta + o( \| \bfDelta \| ),
			\]
			see, e.g., \cite{MagnusNeudecker1999}. By the chain rule, since the differential of inversion taken at $ \matA^{1/2} $ is $ -(\matA^{-1/2} \otimes \matA^{-1/2} ) $, 
			\begin{align*} 
				D \matA^{-1/2} &= D (\matA^{1/2})^{-1} = - (\matA^{-1/2} \otimes \matA^{-1/2} ) D \matA^{1/2} \\
				& = - (\matA^{-1/2} \otimes \matA^{-1/2} ) \vecop^{-1}\left(  (\matid_n \otimes \matA^{1/2} + \matA^{\top/2}  \otimes \matid_n )^{-1} \vecop \matid_N \right).
			\end{align*}
			Therefore,
			\[
			D \bfSigma(F)^{-1/2} = - (\bfSigma^{-1/2}(F) \otimes \bfSigma^{-1/2}(F) ) \vecop^{-1}\left(  (\matid_n \otimes \matA^{1/2} + \matA^{\top/2}  \otimes \matid_n )^{-1} \vecop \matid_N \right).
			\]
			The second assertion again follows from the chain rule. 
		\end{proof}
		
		\begin{proof}[Proof of Theorem~\ref{ThResidualProcess}]
			Decompose the map $ \underline{H}(F) $ as $  \underline{H}(F) = F \circ \psi \circ \varphi (F) $. Here $ \varphi : \calF_2\to \R^{q'} \times \R^{q \times q}$ is defined by 
			\[
			\varphi(F) = \begin{pmatrix}  \bfmu(F) \\ \bfSigma(F) \end{pmatrix}, \qquad F \in \calF_2,
			\]
			with derivative $ \varphi'_F : \calD_{\calF_2,0} \to \R^{q'} \times \R^{q \times q} $ given by
			\[
			\varphi'_F( \Delta ) = \begin{pmatrix} \int_{-\bftau}^{\bftau} \vecu \, d \Delta(\vecu) \\   D \bfSigma(F)^{1/2} \vecop  \bfSigma'_F( \Delta ) \end{pmatrix}, \qquad \Delta \in \calD_{\calF_2,0},
			\]
			and $ \psi $ maps a pair $ (\vecm, \matS ) \in \R^{q'} \times \R^{q \times q}$ to the set $ \text{Aff}(\R^{q'+q}) $ of affine mappings from $ \R^{q'+q} $ to $ \R^{q'+q} $, where the affine transformation $\psi\begin{pmatrix} \vecm \\ \matS \end{pmatrix}$ is defined as
			\[
			\psi\begin{pmatrix} \vecm \\ \vecS \end{pmatrix}\begin{pmatrix} \vecx \\ \vecz \end{pmatrix} = \begin{pmatrix} \vecm + \vecS \vecx \\ \vecz \end{pmatrix}, 
			\qquad \begin{pmatrix} \vecx \\ \vecz \end{pmatrix} \in \R^{q'} \times \R^{q\times q}.
			\] 
			For any direction $ \Delta = ( \Delta_\vecm, \Delta_{\matS} ) \in \R^{q'} \times \R^{q \times q} $ the derivative $ \psi'_{(\vecm, \matS)}( \Delta_\vecm, \Delta_{\matS} ) \in \text{Aff}(\R^{q'+q})  $ is given by
			\[ 
			\psi'_{(\vecm, \matS)}( \Delta_\vecm, \Delta_{\matS} )\begin{pmatrix} \vecx \\ \vecz \end{pmatrix} =
			\begin{pmatrix} \Delta_{\vecm} + \Delta_{\matS} \vecx \\ \vecnull_q \end{pmatrix}, 
			\qquad \begin{pmatrix} \vecx \\ \vecz \end{pmatrix} \in \R^{q'} \times \R^{q\times q}.
			\]
			Thus, denoting the affine transformation by $ A(\Delta_{\vecm},\Delta_{\matS}) $, 
			\[
			\psi'_{(\vecm, \matS)}( \Delta_\vecm, \Delta_{\matS} )
			= \begin{pmatrix} A({\Delta_{\vecm},\Delta_{\matS}}) \\ \vecnull_q \end{pmatrix}, 
			\]
			Since for any c.d.f. $ Q $ we have $ Q'_F(\Delta) = \Delta $, the chain rule yields
			the derivative $ H'_F: \calD_{\calF_2,0} \to \calE $ 
			\begin{align*}
				\underline{H}'_F(\Delta)(\vecx, \vecz) & = \psi'_{\varphi(F)}( \varphi'_F( \Delta) ) \\
				& = A\left( \int_{-\bftau}^{\bftau} \vecu \, d \Delta( \vecu ), D \bfSigma(F)^{1/2} \vecop  \bfSigma'_F( \Delta ) \right).
			\end{align*} 
			Now, since
			\[
			\sqrt{n}( \hat{F}_n(\cdot) - F(\cdot) ) \Rightarrow \calB_F^0(\cdot), \qquad n \to \infty,
			\]
			using the functional delta method for Hadamard differentiable functions, we can conclude that
			the empirical residual process converges weakly to a Gaussian process. Precisely,
			\begin{align*}
				\sqrt{n}( \tilde{F}_{n}(\cdot) - F_{(\vecU,\vecZ)}( \cdot ) 
				& = \sqrt{n}( \underline{H}(\hat{F}_n)( \cdot ) - \underline{H}(F)( \cdot  ) \\
				& \Rightarrow \underline{H}'_F( \calB_F^0 )( \cdot ),
			\end{align*}
			as $n \to \infty $.
		\end{proof}
		
		\begin{proof}[Proof of Theorem~\ref{Th_CLT}]
			We have the representations
			\begin{align*}
				\tilde{c}_{\text{prop}}( z_k ) &= T_k \circ \underline{H}( \hat{F}_n ),\\
				{c}_{\text{prop}}( z_k ) &= T_k \circ \underline{H}( F ),
			\end{align*}
			for $ 1 \le k \le K $. Since 
			\[
			\sqrt{n}( \hat{F}_n - F ) \Rightarrow \calB^0( F ), \qquad n \to \infty,
			\]
			see, e.g., \cite{Kosorok2008}, the functional delta method entails
			\begin{align*}
				\left( \sqrt{n}( \tilde c_{\text{prop}}(z_k ) - c_{\text{prop}}( z_k ) ) \right)_{k=1}^K & = \sqrt{n}( T \circ \underline{H}( \hat F_n ) -  T \circ \underline{H}( F_n )  ) \\
				& \Rightarrow ( T'_{k,\underline{H}(F)}( \underline{H}_F'(\calB^0(F) ) ))_{k=1}^K, \\
				& = ( T'_{k, \underline{H}(F)}( \calB^0_{\bfmu(F),\bfSigma(F)}) )_{k=1}^K 
			\end{align*}
			as $ n \to \infty $. 
		\end{proof}
		
		\begin{proof}[Proof of Theorem~\ref{Th_BTCLT}]
			We have the representations
			\begin{align*}
				\tilde{c}_{\text{prop}}( z_k ) &= T_k \circ \underline{H}( \hat{F}_n ),\\
				\tilde{c}_{\text{prop}}^*( z_k ) &= T_k \circ \underline{H}( \hat{F}_n^* ),
			\end{align*}
			for $ 1 \le k \le K $. Since the bootstrap empirical process satisfies 
			\[
			\sqrt{n}( \hat{F}_n^* - \hat{F}_n ) \Rightarrow \calB^0( F ), \qquad n \to \infty,
			\]
			$P$-outer almost surely, see, e.g., \cite{Kosorok2008}, we can conclude, by virtue of the functional delta method that 
			\begin{align*}
				\left( \sqrt{n}( \tilde c_{\text{prop}}^*(z_k ) - \tilde c_{\text{prop}}( z_k ) ) \right)_{k=1}^K & = \sqrt{n}( T \circ \underline{H}( \hat F_n^* ) -  T \circ \underline{H}( \hat F_n )  ) \\
				& \Rightarrow ( T'_{k,\underline{H}(F)}( \underline{H}_F'(\calB^0(F) ) ))_{k=1}^K, \\
				& = ( T'_{k, \underline{H}(F)}( \calB^0_{\bfmu(F),\bfSigma(F)}) )_{k=1}^K 
			\end{align*}
			as $ n \to \infty $. 
		\end{proof}



\bibliographystyle{chicago}      
\bibliography{lit}   


\end{document}